\documentclass[11pt]{article}

\usepackage[utf8]{inputenc}
\usepackage{
    aliascnt,
    amsmath,
    amsfonts, 
    amsthm, 
    amssymb,
    authblk,
    booktabs,
    color,
    comment,
    dcolumn,
    enumerate,
    fullpage,
    graphics,
    hyperref,
    ifthen, 
    makecell,
    mathrsfs,
    multirow,
    tabularx,
    tikz,
    titlesec,
}
\usepackage[numbers]{natbib}

\usepackage[noend,ruled,noline,linesnumbered]{algorithm2e}
\SetKwProg{ForEach}{for each}{}{}
\SetKwProg{While}{while}{}{}
\SetKwIF{If}{ElseIf}{Else}{if}{}{else if}{else}{}
\SetKw{And}{and}
\SetKw{Goto}{goto}

\makeatletter
\newenvironment{Ualgorithm}[1][htpb]{\def\@algocf@post@ruled{\kern\interspacealgoruled\hrule  height\algoheightrule\kern3pt\relax}%
\def\@algocf@capt@ruled{under}%
\setlength\algotitleheightrule{0pt}%
\begin{algorithm}[#1]}
{\end{algorithm}}
\makeatother

\usetikzlibrary{arrows.meta}

\newcounter{daggerfootnote}

\makeatletter
\newenvironment{varsubequations}[1]
 {%
  \addtocounter{equation}{-1}%
  \begin{subequations}
  \renewcommand{\theparentequation}{#1}
  \def\@currentlabel{#1}%
 }
 {%
  \end{subequations}\ignorespacesafterend
 }
\makeatother

\newcommand\mc[1]{\multicolumn{1}{c}{#1}} 
\newcolumntype{.}{D{.}{.}{-1}}
\makeatletter
\newcolumntype{Z}[3]{>{\mathversion{nxbold}\DC@{#1}{#2}{#3}}c<{\DC@end}}
\makeatother
\newcommand\boldc[2]{\multicolumn{1}{Z{.}{.}{-1}}{\textbf{#1}.\textbf{#2}}} 

\titlespacing
    \section
        {0pt}                     
        {1em plus 1pt minus 1pt}  
        {0pt plus 1pt minus 1pt}  
\titlespacing
    \subsection
        {0pt}
        {1em plus 1pt minus 1pt} 
        {0pt plus 1pt minus 1pt}

\titlelabel{\thetitle.\ }

\DeclareMathVersion{nxbold}
\SetSymbolFont{operators}{nxbold}{OT1}{cmr} {b}{n}
\SetSymbolFont{letters}  {nxbold}{OML}{cmm} {b}{it}
\SetSymbolFont{symbols}  {nxbold}{OMS}{cmsy}{b}{n}

\newtheorem{theorem}{Theorem}

\newaliascnt{lemma}{theorem}
\newaliascnt{proposition}{theorem}
\newaliascnt{corollary}{theorem}
\newaliascnt{observation}{theorem}

\newtheorem{lemma}[lemma]{Lemma}
\newtheorem{proposition}[proposition]{Proposition}
\newtheorem{corollary}[corollary]{Corollary}

\aliascntresetthe{lemma}
\aliascntresetthe{proposition}
\aliascntresetthe{corollary}
\aliascntresetthe{observation}

\newcommand{\N}{\mathbb{N}}
\renewcommand{\emptyset}{\varnothing}
\newcommand{\diam}{\ensuremath{\mathrm{diam}}}

\newcommand{\refp}[1]{(\ref{#1})}
\newcommand{\tr}{\ensuremath^\mathrm{T}}

\DeclareMathOperator{\conv}{conv}
\DeclareMathOperator{\cliquelb}{\ensuremath{\omega^\prime}}

\usetikzlibrary{decorations.pathreplacing}

\hypersetup{
 colorlinks=true,
 linkcolor=blue,
 citecolor=blue,
}

\title{An integer program and new lower bounds for computing the strong rainbow connection numbers of graphs}
\author{{\Large Logan A. Smith, David T. Mildebrath, and Illya V. Hicks}}
\date{\today}
\affil{
    \textit{Department of Computational and Applied Mathematics}\\
    \textit{Rice University}\\
    \textit{Houston, TX 77005}}

\begin{document}
\maketitle

\begin{abstract}\noindent
We present an integer programming model to compute the strong rainbow connection number, $src(G)$, of any simple graph $G$. We introduce several enhancements to the proposed model, including a fast heuristic, and a variable elimination scheme. Moreover, we present a novel lower bound for $src(G)$ which may be of independent research interest.  We solve the integer program both directly and using an alternative method based on iterative lower bound improvement, the latter of which we show to be highly effective in practice. To our knowledge, these are the first computational methods for the strong rainbow connection problem. We demonstrate the efficacy of our methods by computing the strong rainbow connection numbers of graphs containing up to $379$ vertices.
\end{abstract}

\section{Background and Motivation}\label{sec:introduction}

Let $G$ be a non-empty simple connected graph with unweighted edges, and let $c:E(G)\to\{1,\dots,k\}$ for $k\in\N$ be a $k$-coloring of the edges of $G$. 
Note that $c$ is not necessarily a proper coloring, as adjacent edges may be the same color; throughout the paper we will consider colorings in this sense unless otherwise specified. 
We say that edge $e$ is color $k'$ if $c(e) = k'$. 
A path $P$ contained in $G$ is rainbow if $c$ maps each edge in $E(P)$ to a distinct color (i.e., for $e_1,e_2\in{E(P)}$, $e_1\neq{e_2}$ implies $c(e_1)\neq{c(e_2)}$).
The graph $G$ is (strongly) rainbow connected with respect to $c$ if, for every pair of vertices $u,v\in{V(G)}$, there exists a (shortest) $(u,v)$-path $P$ contained in $G$ which is rainbow. The (strong) rainbow connection number $rc(G)$ ($src(G)$) is the minimum number of colors $k$ for which there exists a (strong) rainbow $k$-coloring of $G$. It is known that in general, $\diam(G) \leq rc(G)\leq{src(G)} \leq m$, where $m = |E(G)|$ and $\diam(G)$ is the diameter of $G$.

The concept of (strong) rainbow connection was first introduced by Chartrand et al.~\cite{chartrand2008}, and was originally intended to model the flow of classified information between government agencies in the aftermath of the terrorist attacks of September 11, 2001 \cite{chartrand2009}. Such agencies must communicate in a secure manner due to the sensitivity of their information, but messages must also traverse a complex network of intermediaries with established security protocols. Furthermore, while it is important to have enough security protocols to keep information secure, it is also desirable to have few enough to maintain reasonable simplicity. To address this concern, Chartrand et al.~\cite{chartrand2008} originally posed rainbow connection as the following optimization problem: ``What is the minimum number of passwords or ﬁrewalls needed that allows one or more secure paths between every two agencies so that the passwords along each path are distinct?" Rainbow connection has since been applied in other areas, including the routing of information over secure computer networks (e.g.~rainbow paths arise in~``onion routing'' \cite{reed1998}). In addition to these applications, rainbow connection is of theoretical interest, and has recently garnered significant attention (see Li et al.~\cite{li2013a} for a broad survey of rainbow connection and its many variants).

In spite of this growing attention, there does not currently exist a method for computing $src(G)$ in general graphs $G$. 
Previous works have proposed heuristic methods to produce rainbow colorings for variants of the strong rainbow connection problem~\cite{rc-heur,MOP-heuristic,rvc-heur}. However, none of these methods apply to the strong rainbow connection problem itself. Moreover, even in the event that these methods find an optimal solution, they are unable to provide certificates of optimality.  
In this work, we present the first exact solution method for computing $src(G)$ for general graphs. Our method relies on integer programming, which has a long history of solving difficult combinatorial problems (see, e.g., Nemhauser and Wolsey \cite{nemhauser1988}). The critical component of our method is a practical technique for computing a strong lower bound for $src(G)$. In addition to its computational value, we believe this lower bound is also of theoretical interest. 

The problem of computing $src(G)$ is known to be theoretically challenging.  For certain classes of graphs---including trees, cycles, wheels, complete multipartite graphs \cite{chartrand2008}, fan and sun graphs \cite{sy2013}, stellar graphs \cite{shulhany2016} and block clique graphs \cite{keranen2018}--- $src(G)$ may be computed in polynomial time. However, in general, determining whether $src(G)\leq{k}$ for $k\geq{3}$ is $\mathcal{NP}$-hard, even when $G$ is bipartite \cite{ananth2011} (the same is true of $rc(G)$ \cite{chakraborty2009}). Moreover, simply determining whether or not a given edge coloring strongly rainbow connects a graph is known to be $\mathcal{NP}$-complete in general graphs~\cite{uchizawa2013}.

In this paper we make the following contributions:
\begin{itemize}
\item We present the first exact solution method for computing $src(G)$ for general graphs $G$. 
\item We derive a novel lower bound for $src(G)$ for general graphs $G$. In addition to its value for computing the strong rainbow connection number, we believe that this bound may be of independent research interest.
\item We introduce a series of computational enhancements, including a fast combinatorial heuristic and a variable elimination scheme, which significantly improve the performance of our method.
\item We demonstrate the efficacy of our method by applying it to a collection of graphs containing up to $379$ vertices.
\end{itemize}

The remainder of this paper is organized as follows.
\autoref{sec:prelim} presents notation and preliminaries used throughout the paper. \autoref{sec:lb} introduces a novel lower bound for the strong rainbow connection numbers of general graphs. \autoref{sec:ip-land} proposes an integer programming model to compute the strong rainbow connection numbers of general graphs, as well as a series of computational enhancements to improve the tractability of the integer programming model. \autoref{sec:computation} contains the results of several computational experiments, and concluding remarks are given in \autoref{sec:conclusion}.

\section{Notation and Preliminaries}\label{sec:prelim}

Let $G$ be a graph with vertex set $V(G)$ and edge set $E(G)$. An edge $e=uv\in{E(G)}$ is said to be {\em incident} upon the vertices $u,v$. A graph $G$ is {\em connected} if there exists a path between every pair of vertices in $G$, and {\em simple} if both no vertex is adjacent to itself and no two distinct edges are incident upon the same pair of vertices.
For the remainder of the paper, we assume that all considered graphs are simple, connected, and have at least one edge. If every pair of vertices are adjacent in $G$, then $G$ is called {\em complete}. The complete graph on $n$ vertices is denoted $K_n$. Additionally, a set of vertices which induces a complete graph is called a {\em clique}. The cardinality of the largest clique contained in a graph $G$ is called the {\em clique number} of $G$, and is denoted $\omega(G)$. A function $f:V(G) \to \N$ is called a {\em proper coloring of the vertices} of $G$, or simply a {\em proper coloring} of $G$, if for every pair of adjacent vertices $u,v \in V(G)$, $f(u) \neq f(v)$. The smallest $k \in \N$ for which there exists a proper coloring $f:V(G) \to \{1, \dots, k\}$ is called the {\em chromatic number} of $G$ and is denoted $\chi(G)$. It is well known that $\chi(G) \geq \omega(G)$ holds in general graphs. 

A graph $G$ with $\chi(G) \leq 2$ is called {\em bipartite}, and a {\em complete bipartite graph} is a graph whose vertices can be partitioned into two vertex sets $V_1, V_2$ such that no edge is incident upon two vertices within the same partition and no edge can be added to $G$ that is incident upon vertices in different partitions. The graph $K_{n_1,n_2}$ denotes the complete bipartite graph for which $|V_1| = n_1$ and $|V_2| = n_2$. The {\em distance} between vertices $u, v$, denoted $d(u,v)$, is the the number of edges in any shortest $(u,v)$-path in $G$. A graph is called {\em geodetic} if every pair of distinct vertices in that graph are connected by a unique shortest length path. The union of two graphs $G_1, G_2$ is the graph whose vertex set and edge set are $V(G_1) \cup V(G_2)$ and $E(G_1) \cup E(G_2)$, respectively. Unless otherwise stated, the vertex and edge sets of $G_1, G_2$ are not necessarily assumed to be disjoint. Similarly, we define the intersection of $G_1, G_2$ as the graph with the vertex set $V(G_1) \cap V(G_2)$ and the edge set $E(G_1) \cap E(G_2)$. 

Additional notation and well known results in graph theory are provided by Bondy and Murty~\cite{bondy}.


\section{New Lower Bounds for Strong Rainbow Connection}\label{sec:lb}

In this section, we introduce a novel lower bound for the strong rainbow connection number $src(G)$. This lower bound and its construction will play a critical role in the computational methods introduced later. Additionally, we believe that this bound may be of independent theoretical interest.

By definition, a strong rainbow coloring connects each pair of vertices in $G$ with a rainbow shortest path. Consequently, we begin by studying the set $\mathcal{P}$ of shortest paths in $G$. For each pair of vertices $u\neq{v}\in{V(G)}$, let $\mathcal{P}_{uv}$ denote the set of shortest $(u,v)$-paths in $G$.
To analyze the structure of the set $\mathcal{P}$, it is useful to consider a set of directed graphs, $\mathcal{D}(G) = \{D_{uv} : u \neq v \in V(G)\}$ where $D_{uv}$ is obtained by taking the graph union of each directed shortest $(u,v)$-path in $G$. This set of directed graphs will be used to construct an auxiliary graph which will provide a lower bound for the strong rainbow connection numbers of general graphs.

Additionally, we introduce a notion of separation to indicate edges and vertices which are always present in certain shortest paths. Vertices $u_1,u_2\in{V(G)}$ are {\em separated} by a vertex $v\in{V(G) \backslash \{u_1, u_2\}}$ if $v\in{V(P)}$ for all $P\in\mathcal{P}_{u_1u_2}$. Vertices $u_1,u_2\in{V(G)}$ are {\em separated} by an edge $e\in{E(G)}$ if $e\in{E(P)}$ for all $P\in\mathcal{P}_{u_1 u_2}$. Equivalently, an edge $v_1 v_2$ separates vertices $u_1, u_2$ in $G$ if and only if there is a directed edge $e$ between $v_1,v_2$ in $D_{u_1 u_2}$ and no directed $(u_1,u_2)$-path in $D_{u_1 u_2} \backslash e$. Similarly, a vertex $v$ separates vertices $u_1, u_2$ in $G$ if and only if there is no directed $(u_1,u_2)$-path in $D_{u_1 u_2} - v$. 
\begin{proposition} \label{prop:cutvx-connectivity}
Let $G$ be a graph and $f$ be an edge coloring of $E(G)$. Let $v_1, v_2, u \in V(G)$ such that $u$ separates $v_1$ and $v_2$. If there exists a rainbow shortest path from $v_1$ to $v_2$ in $G$ with respect to $f$, then there exist rainbow shortest paths in $G$ from $v_1$ to $u$ and from $u$ to $v_2$ with respect to $f$. 
\end{proposition}
\begin{proof}
Let $P\in\mathcal{P}_{v_1v_2}$ be rainbow with respect to $f$. Because $u$ separates $v_1,v_2$, by definition $u\in{V(P)}$. Let $P'$ be the $(v_1,u)$-subpath contained in $P$, and let $P''$ be the $(u,v_2)$-subpath contained in $P$. Because $P'$, $P''$ are each subpaths of a shortest path, they must themselves be shortest paths between their respective ends. Additionally, because $P$ is rainbow with respect to $f$, $f$ maps each of the edges in $E(P')$ and $E(P'')$ to distinct colors. Thus, there exist rainbow shortest paths in $G$ from $v_1$ to $u$ and $u$ to $v_2$.
\end{proof}

\begin{figure}[t!]\centering
\begin{tikzpicture}[scale=4.0]
\input{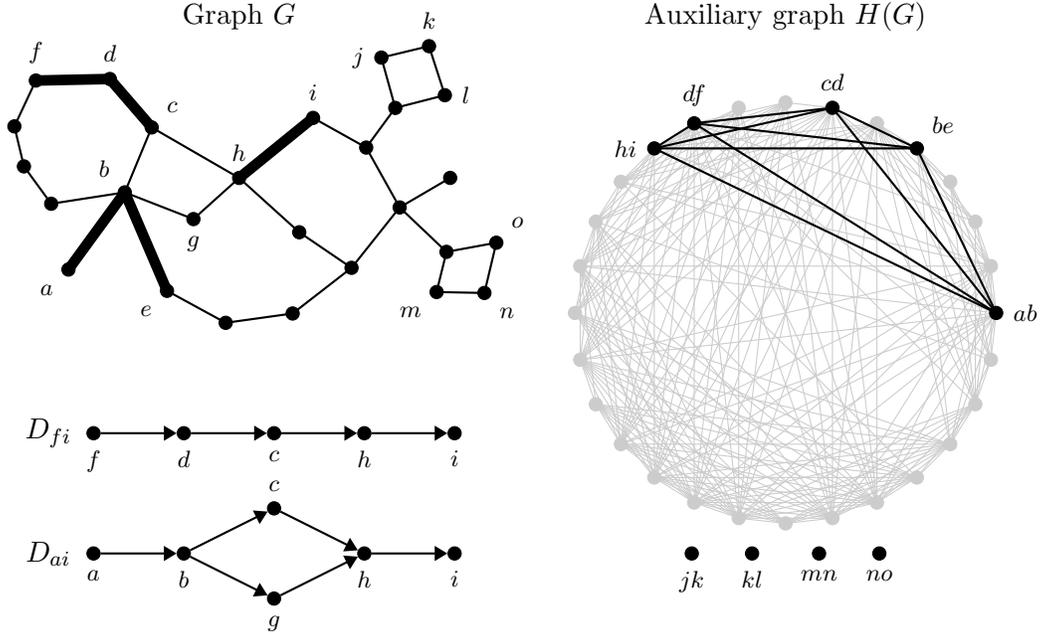}
\end{tikzpicture}
\caption{
Sample graph $G$ (top left) and its auxiliary graph $H= H(G)$ (right). The bold edges $ab$, $be$, $cd$, $df$ and $hi$ in $E(G)$ correspond to the highlighted clique in $H$, and thus must be different colors in any edge coloring that strongly rainbow connects $G$ (\autoref{cor:edges-different}). Two example directed graphs from $\mathcal{D}(G)$ for the vertex pairs $(f,i)$ and $(a,i)$ are shown at bottom left. From the directed graph $D_{ai}$ we see that the edges $ab$ and $hi$ each separate vertices $a$ and $i$, and thus $ab$ and $hi$ are adjacent in $H$. Note that the four edges $jk, kl, mn,$ and $no$ each fail to separate any pairs of vertices in $G$. Their corresponding vertices in $H$ are thus each isolated. 
}
\label{fig:example}
\end{figure}

We now define an auxiliary graph $H(G)$ which encodes information about the set of separating edges in $G$. Let $H(G) = (V_H, E_H)$, where $V_H = E(G)$ and $E_H = \{e_1 e_2 \in E(G) : $ there exists a pair of vertices $ v_1, v_2 \in V(G)$ which are separated by both $e_1$ and $e_2\}$. 
An example of a graph $G$ and its auxiliary graph $H(G)$ is shown is \autoref{fig:example}. 

To ease notation we will often denote $H(G)$ as simply $H$. Additionally, define $\cliquelb(G):=\omega(H)$ and $\chi'(G):=\chi(H)$.
\autoref{thm:src-bound} shows how $H$ can be used to provide lower bounds on the strong rainbow connection number of the graph $G$.

\begin{theorem} \label{thm:src-bound}
For any graph $G$, $src(G) \geq \chi'(G) \geq \cliquelb(G)$. Additionally, if an edge coloring $f$ strongly rainbow connects $G$, then $f$ is a proper coloring of the vertices of $H$. 
\end{theorem}
\begin{proof}

We first verify that the second statement of the theorem holds via its contrapositive. Let $G$ be a graph with edge coloring $f$ and assume that $f$ is not a proper coloring of the vertices of the auxiliary graph $H$. Since $f$ is not a proper coloring of the vertices of $H$, there must exist $e_1, e_2 \in V(H)$ such that $e_1,e_2$ are adjacent in $H$ and $f(e_1) = f(e_2)$. Since $e_1$ and $e_2$ are adjacent in $H$, there must exist a pair of vertices $v_1, v_2 \in V(G)$ such that $e_1$ and $e_2$ each separate $v_1$ and $v_2$ in $G$---that is, each $(v_1,v_2)$-shortest-path in $G$ must include both edges $e_1$ and $e_2$. Because $f(e_1) = f(e_2)$, we conclude that no shortest $(v_1,v_2)$-path in $G$ is rainbow. Thus, the edge coloring $f$ cannot strongly rainbow connect $G$. 

Having established that any strong rainbow edge coloring $f$ of a graph $G$ must also be a proper coloring of the vertices of $H$, it follows that no edge coloring $f'$ of $G$ that uses fewer than $\chi'(G)$ colors can strongly rainbow connect $G$ as then $f'$ would also be a proper coloring of the vertices of $H$ despite having fewer than $\chi'(G)$ colors. We conclude then that $src(G) \geq \chi'(G)$. Additionally as $\chi'(G) \geq \omega'(G)$, the full statement $src(G) \geq \chi'(G) \geq \cliquelb(G)$ holds. 
\end{proof}


\begin{corollary}\label{cor:edges-different}
Let $G$ be a graph and $f$ be a strong rainbow edge coloring of $G$. If $U \subseteq V(H) = E(G)$ is a clique in $H$, then $f(e_1) \neq f(e_2)$ for any $e_1, e_2 \in U$. 
\end{corollary}
\begin{proof}
By \autoref{thm:src-bound}, if $f$ strongly rainbow connects $G$ then $f$ must also be a proper coloring of $V(H)$. Let $U \subseteq V(H) = E(G)$ be a clique in $H$. Then every pair of edges $e_1, e_2 \in U$ are adjacent in $H$, and so $f(e_1) \neq f(e_2)$.  
\end{proof}

An illustration of \autoref{cor:edges-different} is shown in \autoref{fig:example}.

\begin{theorem} \label{thm:src-geodetic}
For any geodetic graph $G$, $src(G) = \chi'(G)$. Additionally, an edge coloring $f$ strongly rainbow connects $G$ if and only if $f$ is a proper coloring of the vertices of $H$. 
\end{theorem}
\begin{proof}

From \autoref{thm:src-bound}, if $f$ strongly rainbow connects $G$ then $f$ must be a proper coloring of the vertices of $H$ and $src(G) \geq \chi'(G)$. To verify the other direction of the theorem, we now assume that $f$ is a proper coloring of the vertices of $H$. Consider any pair of distinct vertices $u, v$ in $V(G)$. Because $G$ is geodetic, there exists a unique $(u,v)$-shortest-path in $G$---call it $P_{uv}$. Because $P_{uv}$ is unique, it follows that each edge $e\in{E(P_{uv})}$ separates $u$ from $v$, and thus every pair of edges in $E(P_{uv})$ are adjacent in $H$. Since $f$ is a proper coloring of $V(H)$, $f$ must map each pair of edges in $E(P_{uv})$ to different colors, thus the unique $(u,v)$-shortest-path in $G$ must be rainbow with respect to $f$. Because $u,v$ were chosen arbitrarily in $V(G)$ it follows that every pair of vertices in $G$ are connected by a rainbow shortest path, and thus that $f$ strongly rainbow connects $G$.

Because every proper coloring of the vertices of $H$ must strongly rainbow connect $G$, there must exist a proper vertex coloring $f'$ which colors the vertices of $H$ using exactly $\chi(H)$ colors. Since the function $f'$ is also a strong rainbow coloring of $G$ we conclude that $src(G) \leq \chi'(G)$. Combining this with the earlier inequality given in \autoref{thm:src-bound}, we have $src(G) = \chi'(G)$.
\end{proof}

Despite the fact that $\chi'(G)$ is a theoretically tighter bound than $\cliquelb(G)$, in \autoref{sec:computation}, we compute only $\cliquelb(G)$, not $\chi'(G)$. The reason for this choice is two-fold. First, for many of the graphs we consider---particularly those describing social and infrastructure networks---the clique bound satisfies $\cliquelb(G)=src(G)$. In such situations, \autoref{thm:src-bound} implies that $\chi'(G)=\cliquelb(G)$, and thus computing $\chi'(G)$ is not necessary. Second, computing $\chi'(G)$ requires computing the chromatic number of $H$, while computing $\cliquelb(G)$ requires computing the clique number of $H$.
Although computing both of these values is $\mathcal{NP}$-hard in general, computational experiments performed by Verma et al.~\cite{verma} suggest that in sparse graphs, clique numbers can often be computed much more quickly than chromatic numbers. Verma et al. also propose and consider several lower bounds~\cite{culberson1996exploring} for the chromatic numbers of sparse graphs, and found that the clique number provides the best lower bound in 45 of the 53 instances they consider. In all, we observe that computing $\cliquelb(G)$ sacrifices very little in terms of bound quality and is relatively inexpensive in terms of overall runtime.

Because of this observation, we devote the remainder of this section to studying the lower bound $\omega'(G)$.
Specifically, we show that the lower bounds $\diam(G)$ and $\cliquelb(G)$ for $src(G)$ are incomparable (\autoref{prop:incomparable}) and may be simultaneously arbitrarily far from $src(G)$ (\autoref{prop:arb-far}).
In spite of these results, we demonstrate in \autoref{sec:computation} that the bound provided by $\cliquelb(G)$ is stronger than that of $\diam(G)$ in many random graphs distributed by the Erd\H{o}s-R\'enyi random graph model \cite{ER-graphs}, as well as many graphs commonly used to describe infrastructure networks or social networks. 

\begin{figure}\centering
\begin{tikzpicture}

\tikzstyle{whitenode} = [draw,circle,minimum size=5pt,inner sep=0pt, thick]
\tikzstyle{blacknode} = [draw,circle,minimum size=5pt,inner sep=0pt, fill]
\tikzstyle{dotedge} = [very thick, dotted]
\tikzstyle{solidedge} = [thick] 


\newcommand{\DrawK}[4]{ 
    \pgfmathsetmacro{\k}{#1}
    \pgfmathsetmacro{\gap}{#2}
    \pgfmathsetmacro{\x}{#3}
    \pgfmathsetmacro{\y}{#4}
    \pgfmathsetmacro{\kpt}{\k+1}
    \foreach \i in {0,...,\kpt} {
        \node[blacknode] at (\x+\i*\gap,\y+\gap) {};
        \node[blacknode] at (\x+\i*\gap,\y) {};
        \draw[solidedge] (\x+\i*\gap,\y) -- (\x+\i*\gap,\y+\gap);
        \ifthenelse
        { 
            \i < \kpt
        }
        { 
            \draw[solidedge] (\x+\i*\gap,\y) -- (\x+\gap*\i+\gap,\y);
            \draw[solidedge] (\x+\i*\gap,\y+\gap) -- (\x+\gap*\i+\gap,\y+\gap);
            \draw[solidedge] (\x+\i*\gap,\y) -- (\x+\gap*\i+\gap,\y+\gap);
            \draw[solidedge] (\x+\i*\gap,\y+\gap) -- (\x+\gap*\i+\gap,\y);
        }
        {
        }
    }
} 

\DrawK{1}{1.0}{0}{0}
\DrawK{2}{1.0}{3}{0}
\DrawK{3}{1.0}{7}{0}

\node[] at (1,-0.5) {$k=1$};
\node[] at (4.5,-0.5) {$k=2$};
\node[] at (9,-0.5) {$k=3$};

\end{tikzpicture}
\caption{Sequence of graphs for which $\diam(G)-\omega'(G)=k$, shown for $k=1,2,3$. Note that the $k^\mathrm{th}$ graph has diameter $k+1$.}
\label{fig:k4s}
\end{figure}
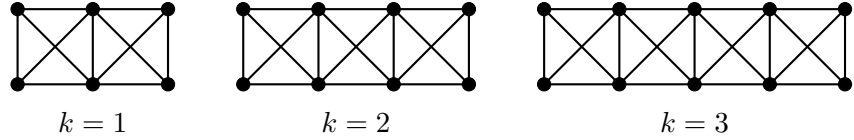

\begin{proposition}\label{prop:incomparable}
For any $k \in \N$, there exist graphs $G$ and $G'$ such that $\diam(G) - \cliquelb(G) = k$ and $\cliquelb(G') - \diam(G') = k$. 
\end{proposition}
\begin{proof}
We first show that for any $k \in \N$ there exists a graph $G$ such that $\diam(G) - \cliquelb(G) = k$. Fix $k \in \N$, and let $G$ be the graph obtained by taking the union of $k+1$ distinct $K_4$ graphs, where the vertices the $i^{th}$ $K_4$ are labeled $v_{2(i-1)+j}$ for $j = 1, 2, 3, 4$. The first few graphs in this sequence are illustrated in  \autoref{fig:k4s}. It can be observed that $\diam(G) = k+1$. Additionally, each pair of non-adjacent vertices in $V(G)$ are connected by at least two edge disjoint shortest paths, thus no pair of non-adjacent vertices in $V(G)$ are separated by an edge. Since a pair of vertices in $V(G)$ can only be separated by an edge which is incident upon both of them, no pair of edges can separate the same pair of vertices. $E(H(G)) = \emptyset$, so $\cliquelb(G) = 1$ and $\diam(G) - \cliquelb(G) = k$ as desired.

To see that for any $k \in \N$ there also exists a graph $G'$ such that $\cliquelb(G') - \diam(G') = k$, fix some $k$ and let $G' = K_{1,k+2}$. Additionally, let $H' = H(G')$ and note that $|E(G')|=k+2$. Since each pair of edges in $E(G')$ form the unique shortest path between the leaves those edges are incident upon, $H' = K_{k+2}$. Thus $\cliquelb(G') = k+2$. Since $\diam(G') = 2$, we have that $\cliquelb(G') - \diam(G') = k$. 
\end{proof}

In the following proposition we make use of a formula shown by Chartrand et al.~\cite{chartrand2008}: for integers $s,t$ such that $1 \leq s \leq t$, $src(K_{s, t}) = \lceil \sqrt[s]{t} \rceil$.

\begin{proposition}\label{prop:arb-far}
For any $k \in \N$, there exist a graph $G$ such that $src(G) - \max(\diam(G), \cliquelb(G)) \geq k$.
\end{proposition}
\begin{proof}
Fix $k \in \N$, and let $G = K_{2,(k+2)^2}$. Then $\diam(G) = 2$, and since there are at least two edge disjoint shortest paths between each pair of vertices in $V(G)$, no pair of vertices in $H(G)$ are adjacent and so $\cliquelb(G) = 1$. Thus
$ src(G) - \max(\diam(G), \cliquelb(G)) = \big\lceil \sqrt[2]{(k+2)^2} \big\rceil - 2 = k. $   
\end{proof}


\section{Model and Enhancements} \label{sec:ip-land}

In this section, we introduce an integer program for computing $src(G)$ for general graphs $G$. We then describe a series of enhancements which leverage the results of \autoref{sec:lb} in order to improve the efficiency of the model in practice.

The remainder of the section is organized as follows. 
We present the integer programming model in  \autoref{subsec:model}.
\autoref{subsec:imp} describes a practical method for constructing the auxiliary graph $H$ introduced in \autoref{sec:lb}. 
\autoref{subsec:ip-imp} explains a method which uses the auxiliary graph $H$ to eliminate both variables and constraints from the integer programming model.
In \autoref{subsec:heuristic}, we present the first heuristic for the (strong) rainbow connection problem, and explain how to use the heuristic to further reduce the size of the integer programming formulation.
\autoref{subsec:bottom-up} describes a novel iterative scheme for computing $src(G)$ by combining our integer programming formulation with the lower bound of \autoref{sec:lb}. Finally, 
in \autoref{subsec:ip-ext} we propose several extensions of our model to different variants of the strong rainbow connection problem.

%

\subsection{An Integer Programming Model}\label{subsec:model}

Let $K=\{1,\dots,K_0\}$ be a set of colors for some $K_0\leq{m}$. We introduce the following binary variables:
\begin{itemize}
\item $x_{ek}=1$ if and only if edge $e\in{E(G)}$ is color $k\in{K}$.
\item $y_P=1$ if and only if path $P\in\mathcal{P}$ is rainbow (i.e.~all edges $e\in{E(P)}$ are different colors).
\item $z_k=1$ if and only if there exists an edge $e\in{E(G)}$ of color $k\in{K}$.
\end{itemize}

The following integer program computes the strong rainbow connection number of $G$, and a corresponding strong rainbow coloring.

\begin{varsubequations}{IP-1}\label{ip:master}
\begin{align}
\label{ip:master-obj}
z^*(G)=\min\;&\sum_{k\in{K}}z_k\\
\label{ip:master-sum-one}
\mathrm{s.t.}\;&\sum_{k\in{K}}x_{ek}=1&&\forall\;e\in{E(G)}\\
\label{ip:master-big-m}
&\sum_{e\in{E(P)}}x_{ek}+(|P|-1)y_P\leq|P|&&\forall\;P\in\mathcal{P},\;\forall\;k\in{K}\\
\label{ip:master-sum-y}
&\sum_{P\in\mathcal{P}_{uv}}y_P\geq1&&\forall\;u\neq{v}\in{V(G)}\\
\label{ip:master-vub}
&x_{ek}\leq{z_k}&&\forall\;e\in{E(G)},\;\forall\;k\in{K}\\
\label{ip:master-symmetry}
&z_k\geq{z_{k+1}}&&\forall\;k\in{K}\setminus\{K_0\}\\
&x_{ek},y_P,z_k\in\{0,1\}&&\forall\;e\in{E(G)},\;\forall\;k\in{K},\;\forall\;P\in\mathcal{P}
\end{align}
\end{varsubequations}

The objective \refp{ip:master-obj} minimizes the total number of colors used in the strong rainbow coloring. The constraints \refp{ip:master-sum-one} through \refp{ip:master-symmetry} are summarized as follows:
\begin{itemize}
\item The constraints \refp{ip:master-sum-one} require that each edge $e\in{E(G)}$ is assigned exactly one color. 
\item The constraints \refp{ip:master-big-m} enforce the logical requirement that if $y_P=1$ for some path $P\in\mathcal{P}$, then at most one edge in path $P$ is assigned color $k\in{K}$.
\item The constraints \refp{ip:master-sum-y} require that, for each pair of distinct vertices $u,v\in{V(G)}$, there exists a $(u,v)$-shortest-path $P$ such that $y_P=1$. 
\item The constraints \refp{ip:master-vub} enforce that if there exists an edge $e\in{E(G)}$ of color $k\in{K}$, then $z_k=1$ (i.e.~color $k\in{K}$ is counted in the objective value \refp{ip:master-obj}).
\item Constraints \refp{ip:master-symmetry} are symmetry breaking constraints which require the number of colors used in the strong rainbow coloring to be consecutive starting at $k=1$. While these constraints are not necessary to ensure a correct model, we find that they have significant computational benefits.
\end{itemize}
The following proposition certifies the correctness of the model \refp{ip:master}.

\begin{proposition}
For any graph $G$, $z^*(G)=src(G)$. Moreover, if $(x,y,z)$ is a feasible solution to \refp{ip:master}, then the coloring $f:E\to{K}$ with $f(e)=\sum_{k\in{K}}kx_{ek}$ (i.e.,~which maps an edge $e$ to the unique $k$ with $x_{ek}=1$) strongly rainbow connects $G$.
\end{proposition}

We note two relevant features of the integer program \refp{ip:master}. First, \refp{ip:master} utilizes a total of $|\mathcal{P}| + (m + 1)|K|$ binary variables and $(|\mathcal{P}| + m + 1)|K| + \binom{n}{2} + m - 1$ constraints. Unfortunately, $O(2^n)$ is a tight upper bound for $|\mathcal{P}|$ in general graphs. Thus at worst, the integer program \refp{ip:master} may have an exponential number of rows and columns. In spite of the potentially large size of \refp{ip:master}, we demonstrate in \autoref{sec:computation} that for graphs appearing in several applications, the cardinality of the set $\mathcal{P}$ is sufficiently small as to render \refp{ip:master} computationally tractable. Second, \refp{ip:master} is feasible if and only if $K_0\geq{src(G)}$. Consequently, building a feasible model requires a valid upper bound on the strong rainbow connection number $src(G)$. While the na\"ive choice of $K_0=m$ is valid for any graph $G$, we introduce a heuristic in \autoref{subsec:heuristic} which computes a much tighter upper bound, thus reducing the size of integer program \refp{ip:master} and improving its computational performance.

In addition, by replacing $\mathcal{P}$ with the set of \textit{all} paths in $G$ (rather than the set of \textit{shortest} paths), \refp{ip:master} can be used to compute the rainbow connection number $rc(G)$. In general, the total number of paths in a graph $G$ is much larger than the number of shortest paths in $G$. Moreover, as we discuss throughout the paper, the set $\mathcal{P}$ of shortest paths contains structure which can be exploited for computation. 
Consequently, we do not consider the computation of $rc(G)$ any further.

\subsection{Constructing the Auxiliary Graph}
\label{subsec:imp}

In this section, we explain how to construct the auxiliary graph $H$ introduced in \autoref{sec:lb}.
The techniques described in this section will be used throughout the remainder of the paper to improve the computational performance of the model \refp{ip:master}.
The construction of $H$ proceeds in three steps. First, for each $u\in{V(G)}$ we construct a directed graph $D_u$, which is related to the directed graph $D_{uv}$ introduced in \autoref{sec:lb}. Second, we use the set directed graphs $\{ D_u: u \in V(G)\}$ to compute the set of separating edges in $G$. Third, we use the set of separating edges to construct the auxiliary graph $H$. We will show how this construction also allows us to compute the set of separating vertices in $G$, as well as $r_{uv}=|\mathcal{P}_{uv}|$, the number distinct shortest $(u,v)$-paths in $G$. These objects will be used in the computational enhancements introduced in \autoref{subsec:ip-imp} and \autoref{subsec:heuristic}, respectively.

Although they play a helpful role in the analysis of \autoref{sec:lb}, it is not necessary to explicitly compute all $n(n-1)$ directed graphs $D_{uv}$.
Instead, we compute a single auxiliary directed graph $D_u$ for each $u\in{V(G)}$, defined as follows. Let $D_u=(V(G),E_D)$ where the set of directed edges $E_D=\{(v_1,v_2):v_1v_2\in{E(G)}\text{ and }d(u,v_1)+1=d(u,v_2)\}$. 
For fixed $u$, the directed graph $D_u$ can be constructed in $O(n+m)$ time by modifying the breadth first search algorithm to allow a vertex to be discovered by multiple vertices of the previous depth. An example of the directed graphs $D_u$ and their relation to the directed graphs $D_{uv}$ is shown in \autoref{fig:dag-example}.

\begin{figure}[t]\centering
\begin{tikzpicture}[scale=.90]
\definecolor{lgray}{HTML}{A1A0A0}

\tikzstyle{blacknode} = [draw,circle,minimum size=5pt,inner sep=0pt, fill]
\tikzstyle{whitenode} = [draw,circle,minimum size=5pt,inner sep=0pt, thick]
\tikzstyle{lightnode} = [draw,circle,minimum size=5pt,inner sep=0pt, fill, color=lgray]
\tikzstyle{dotedge} = [very thick, dotted]
\tikzstyle{solidedge} = [thick]
\tikzstyle{boldedge} = [line width=4pt]
\tikzstyle{lightedge} = [color=lgray]
\tikzstyle{soliddiredge} = [thick, -Triangle]
\tikzstyle{bolddiredge} = [line width=4pt, -Triangle]
\tikzstyle{lightdiredge} = [color=lgray, -Triangle]

\newcommand{\DrawG}[3]{
    \pgfmathsetmacro{\x}{#1} 
    \pgfmathsetmacro{\y}{#2} 
    \pgfmathsetmacro{\h}{#3} 

    \foreach \i in {1,...,5} { 
        \node[blacknode] (a\i) at (\x,\y+\i*\h-\h) {};
    }
    \foreach \i in {1,...,4} { 
        \node[blacknode] (b\i) at (\x+\h,\y+\i*\h) {};
    }
    \foreach \i in {1,...,3} { 
        \node[blacknode] (c\i) at (\x+2*\h,\y+2*\i*\h-2*\h) {};
    }

    \node[label=left:{$u$}] at (a4) {};
    \node[label=right:{$v_1$}] at (c3) {};
    \node[label=right:{$v_2$}] at (b1) {};
    \node[label=above:{$w$}] at (b4) {};

    \draw[solidedge] (a1) -- (a2);
    \draw[solidedge] (a2) -- (a3);
    \draw[solidedge] (a3) -- (a4);
    \draw[solidedge] (a4) -- (a5);
    \draw[solidedge] (b1) -- (b2);
    \draw[solidedge] (b2) -- (b3);
    \draw[solidedge] (b3) -- (b4);
    \draw[solidedge] (c1) -- (c2);
    \draw[solidedge] (c2) -- (c3);
    \draw[solidedge] (a1) -- (c1);
    \draw[solidedge] (a2) -- (b1);
    \draw[solidedge] (a3) -- (b2);
    \draw[solidedge] (b2) -- (c2);
    \draw[solidedge] (a4) -- (b3);
    \draw[solidedge] (a5) -- (b4);
    \draw[solidedge] (b4) -- (c3);
    \draw[solidedge] (a2) -- (b2);
    \draw[solidedge] (c1) -- (b1);
} 

\newcommand{\DrawD}[3]{
    \pgfmathsetmacro{\x}{#1} 
    \pgfmathsetmacro{\y}{#2} 
    \pgfmathsetmacro{\h}{#3} 

    \node[blacknode] (a1) at (\x,\y) {}; 
    \foreach \i in {1,...,3} { 
        \node[lightnode] (b\i) at (\x+\h,\y+\i*\h-2*\h) {};
    }
    \foreach \i in {1,...,3} { 
        \node[lightnode] (c\i) at (\x+2*\h,\y+\i*\h-2*\h) {};
    }
    \foreach \i in {1,...,4} { 
        \node[lightnode] (d\i) at (\x+3*\h,\y+\i*\h-2.5*\h) {};
    }
    \node[lightnode] (e1) at (\x+4*\h,\y-0.5*\h) {}; 

    \node[blacknode] at (b2) {};
    \node[blacknode] at (b3) {};
    \node[blacknode] at (c3) {};
    \node[blacknode] at (d4) {};

    \node[label={[text=black]above:{\footnotesize$1$}}] at (a1) {};

    \node[label={[text=lgray]above:{\footnotesize$1$}}] at (b1) {};
    \node[label={[text=black]above:{\footnotesize$1$}}] at (b2) {};
    \node[label={[text=black]above:{\footnotesize$1$}}] at (b3) {};

    \node[label={[text=lgray]above:{\footnotesize$1$}}] at (c1) {};
    \node[label={[text=lgray]above:{\footnotesize$2$}}] at (c2) {};
    \node[label={[text=black]above:{\footnotesize$2$}}] at (c3) {};

    \node[label={[text=lgray]below:{\footnotesize$1$}}] at (d1) {};
    \node[label={[text=lgray]above:{\footnotesize$3$}}] at (d2) {};
    \node[label={[text=lgray]above:{\footnotesize$2$}}] at (d3) {};
    \node[label={[text=black]above:{\footnotesize$2$}}] at (d4) {};

    \node[label={[text=lgray]below:{\footnotesize$6$}}] at (e1) {};

    \node[label=left:{$u$}] at (a1) {};
    \node[label=right:{$v_1$}] at (d4) {};
    \node[label={[text=lgray]below:{$v_2$}}] at (d2) {};
    \node[label={[label distance=-2pt]330:$w$}] at (c3) {};

    \draw[lightdiredge] (a1) -- (b1);
    \draw[soliddiredge] (a1) -- (b2);
    \draw[soliddiredge] (a1) -- (b3);

    \draw[lightdiredge] (b1) -- (c1);
    \draw[lightdiredge] (b1) -- (c2);
    \draw[lightdiredge] (b2) -- (c2);
    \draw[soliddiredge] (b2) -- (c3);
    \draw[soliddiredge] (b3) -- (c3);

    \draw[lightdiredge] (c1) -- (d1);
    \draw[lightdiredge] (c1) -- (d2);
    \draw[lightdiredge] (c2) -- (d2);
    \draw[lightdiredge] (c2) -- (d3);
    \draw[soliddiredge] (c3) -- (d4);

    \draw[lightdiredge] (d1) -- (e1);
    \draw[lightdiredge] (d2) -- (e1);
    \draw[lightdiredge] (d3) -- (e1);
} 

\newcommand{\DrawDNew}[3]{
    \pgfmathsetmacro{\x}{#1} 
    \pgfmathsetmacro{\y}{#2} 
    \pgfmathsetmacro{\h}{#3} 

    \node[blacknode] (a1) at (\x,\y) {}; 
    \foreach \i in {1,...,3} { 
        \node[lightnode] (b\i) at (\x+\h,\y+\i*\h-2*\h) {};
    }
    \foreach \i in {1,...,3} { 
        \node[lightnode] (c\i) at (\x+2*\h,\y+\i*\h-2*\h) {};
    }
    \foreach \i in {1,...,4} { 
        \node[lightnode] (d\i) at (\x+3*\h,\y+\i*\h-2.5*\h) {};
    }
    \node[lightnode] (e1) at (\x+4*\h,\y-0.5*\h) {}; 

    \node[blacknode] at (b1) {};
    \node[blacknode] at (b2) {};
    \node[blacknode] at (c1) {};
    \node[blacknode] at (c2) {};
    \node[blacknode] at (d2) {};

    \node[label={[text=black]above:{\footnotesize$1$}}] at (a1) {};

    \node[label={[text=black]above:{\footnotesize$1$}}] at (b1) {};
    \node[label={[text=black]above:{\footnotesize$1$}}] at (b2) {};
    \node[label={[text=lgray]above:{\footnotesize$1$}}] at (b3) {};

    \node[label={[text=black]above:{\footnotesize$1$}}] at (c1) {};
    \node[label={[text=black]above:{\footnotesize$2$}}] at (c2) {};
    \node[label={[text=lgray]above:{\footnotesize$2$}}] at (c3) {};

    \node[label={[text=lgray]below:{\footnotesize$1$}}] at (d1) {};
    \node[label={[text=black]above:{\footnotesize$3$}}] at (d2) {};
    \node[label={[text=lgray]above:{\footnotesize$2$}}] at (d3) {};
    \node[label={[text=lgray]above:{\footnotesize$2$}}] at (d4) {};

    \node[label={[text=lgray]below:{\footnotesize$6$}}] at (e1) {};

    \node[label=left:{$u$}] at (a1) {};
    \node[label={[text=lgray]right:{$v_1$}}] at (d4) {};
    \node[label={[text=black]below:{$v_2$}}] at (d2) {};
    \node[label={[text=lgray,label distance=-2pt]330:$w$}] at (c3) {};

    \begin{scope}[every path/.style=lightdiredge]
        \draw (a1) -- (b3);
        \draw (b2) -- (c3);
        \draw (b3) -- (c3);
        \draw (c3) -- (d4);
        \draw (c1) -- (d1);
        \draw (c2) -- (d3);
        \draw (d1) -- (e1);
        \draw (d2) -- (e1);
        \draw (d3) -- (e1);
    \end{scope}

    \begin{scope}[every path/.style=soliddiredge]
        \draw (a1) -- (b1);
        \draw (a1) -- (b2);
        \draw (b1) -- (c1);
        \draw (b1) -- (c2);
        \draw (b2) -- (c2);
        \draw (c1) -- (d2);
        \draw (c2) -- (d2);
    \end{scope}

} 

\DrawG{0}{0}{1.1}
\DrawD{3.3}{2.2}{1.3}
\DrawDNew{9.5}{2.2}{1.3} 
\end{tikzpicture}
\caption{An example of a graph $G$ (left), the directed graph $D_u$ (center and right), $D_{uv_1}$ (center, bold), and $D_{uv_2}$ (right, bold). Each vertex $v\in{V(D_u)}$ corresponds to a vertex in $V(G)$, and is labeled with $r_{uv}$, the number of shortest $(u,v)$-paths in $G$. For this example, passing $u$ and $v_1$ to \autoref{alg:preprocessing} returns $V_\mathrm{sep}=\{w\}$, $E_\mathrm{sep}=\{wv_1\}$ and $r_{uv_1}=2$.}
\label{fig:dag-example}
\end{figure}
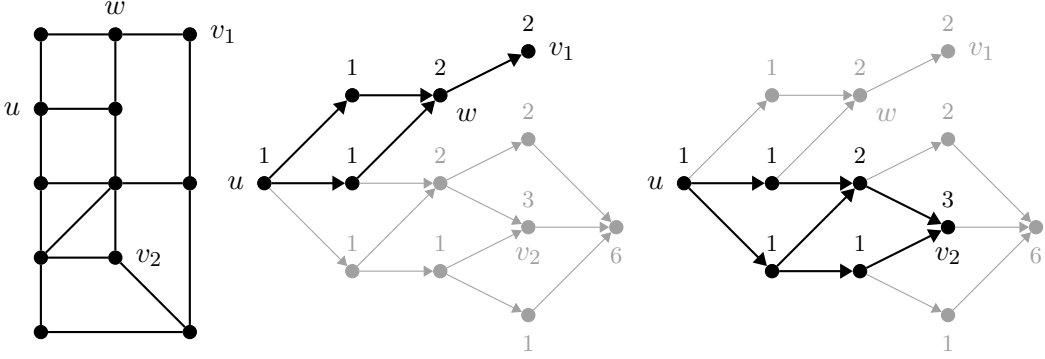

\begin{Ualgorithm}[ht!]
\DontPrintSemicolon
\KwData{Distinct, non-adjacent vertices $u,v\in{V(G)}$, and the directed graph $D_u$.}
\KwResult{$E_\mathrm{sep}$ (the set of edges separating $u$ and $v$); $V_\mathrm{sep}$ (the set of vertices separating $u$ and $v$); $r_{uv}$ (the number of distinct shortest $(u,v)$-paths in $G$).}
$L\gets\{v\}$\;
$r_v\gets1$\;
$V_\mathrm{sep}\gets\emptyset$\;
$E_\mathrm{sep}\gets\emptyset$\;
\While{$u\not\in{L}$}{
    $\displaystyle N\gets\bigcup_{j\in{L}}N^+(j)$\;
    \ForEach{$j\in{N}$\label{line:r1}}{
        $\displaystyle r_j\gets\sum_{i\in{L}\cap{N^-(j)}}r_i$\;\label{line:r2}
    }
    \If{$|L|=|N|=1$}{
        $e'\gets$ the directed edge in $E(D_u)$ connecting the unique element of $N$ to the unique element of $L$\;
        $E_\mathrm{sep}\gets E_\mathrm{sep}\cup\{e'\}$\;
    }
    \If{$|N|=1$ \And $u\not\in{N}$}{
        $V_\mathrm{sep}\gets V_\mathrm{sep}\cup{N}$\; 
    }
    $L\gets{N}$\;
}
\Return{$V_\mathrm{sep}$, $E_\mathrm{sep}$, $r_u$}
\caption{Algorithm to compute the separating edges and vertices for a fixed pair of distinct, non-adjacent vertices $u$ and $v$. If lines \ref{line:r1} and \ref{line:r2} are included, the algorithm also counts $r_{uv}$, the number of distinct shortest $(u,v)$-paths in $G$. Upon termination, for each $j\in{V(G)}$, the value $r_j$ equals the number of shortest $(j,v)$-paths in $G$. For each $j\in{V(G)}$, the set $N^+(j)$ (resp.~$N^-(j)$) is the set of in-neighbors (resp.~out-neighbors) of $j$ in the directed graph $D_u$.}
\label{alg:preprocessing}
\end{Ualgorithm}

Given the directed graph $D_u$ for some $u\in{V(G)}$,
\autoref{alg:preprocessing} computes the set of edges $E_\mathrm{sep}$ which separate $u$ from any non-adjacent vertex $v\in{V(G)}\setminus\{u\}$. As a byproduct, the algorithm also computes the set of separating vertices $V_\mathrm{sep}$, and the  number $r_{uv}=|\mathcal{P}_{uv}|$ of shortest $(u,v)$-paths in $G$. For a fixed pair of vertices $u$ and $v$, the sets $E_\mathrm{sep}$ and $V_\mathrm{sep}$ can be computed in $O(n+m)$ time using \autoref{alg:preprocessing} with lines \ref{line:r1} and \ref{line:r2} removed.
Computing $r_{uv}$ by adding lines \ref{line:r1} and \ref{line:r2} increases the time complexity of \autoref{alg:preprocessing} to $O(n^2+nm)$.

%
Finally, we construct the graph $H$. For all $v_1\neq{v_2}\in{V(G)}$, let $E'_{v_1v_2}$ denote the set of edges which separate $v_1,v_2$ in $G$ (computed using \autoref{alg:preprocessing} in $O(n+m)$ time). For each pair $e_1,e_2$ of distinct elements of $E'_{v_1v_2}$, we add the undirected edge $e_1e_2$ to $H$. For a fixed pair of vertices $v_1,v_2$, $|E'_{v_1v_2}|=O(m^2)$, and thus $H$ can be constructed in $O(n^2(n+m^2))$ time.

\subsection{Variable and Constraint Elimination} \label{subsec:ip-imp}

The results introduced in \autoref{sec:lb} suggest two methods by which we can eliminate variables and constraints from the model \refp{ip:master}. First, by \autoref{cor:edges-different}, for any single $U\subseteq{V(H)}=E(G)$ that is a clique in $H$, we may enforce \textit{a priori} that each edge in $U$ is a different color---i.e., we fix the variables $x_{ek}$ for all $e\in{U}$ and $k\in{K}$. Naturally, this technique eliminates the largest number of variables when $U$ is a maximum clique in $H$, corresponding to the bound $\cliquelb(G)$. In addition, this choice also helps to break symmetry inherent to the model.

Second, per \autoref{prop:cutvx-connectivity}, if a vertex $u$ separates a pair of vertices $v_1,v_2$ and at least one $(v_1,v_2)$-shortest-path is rainbow with respect to an edge coloring $f$, then there must exist shortest paths connecting both $v_1, u$ and $u, v_2$ which are rainbow with respect to $f$. Thus as long as it is enforced that the edge coloring indicated by the $x_{ek}$ variables strongly rainbow connects vertices $v_1, v_2$, it is not necessary to explicitly enforce that $v_1, u$ and $u, v_2$ are connected by rainbow shortest paths. Hence, we may eliminate the variables $y_P$ and constraints \refp{ip:master-big-m} corresponding to paths $P\in\mathcal{P}_{v_1u}\cup\mathcal{P}_{uv_2}$, as well as the constraints \refp{ip:master-sum-y} corresponding to the vertex pairs $v_1, u$ and $u, v_2$.
In \autoref{sec:computation} we provide empirical evidence that this step substantially reduces the size of the formulation \refp{ip:master} for some classes of graphs.

\subsection{A Fast Random Heuristic for (Strong) Rainbow Connection}\label{subsec:heuristic}

To the best knowledge of the authors, no heuristic method has been proposed for strong rainbow connection in general graphs. Polynomial heuristic methods have been proposed for several related problems however, including the rainbow vertex connection problem~\cite{rvc-heur} and the rainbow connection problem, both in general graphs \cite{rc-heur} and in the special case of maximal outer planer graphs~\cite{MOP-heuristic}. Because the upper bounds provided by heuristic methods can often be used to significantly assist with computational methods, it is desirable to also have a heuristic method for strong rainbow connection. Unfortunately, since the previously mentioned heuristics either solve related problems or determine upper bounds on $rc(G)$, which is a lower bound on $src(G)$, they are not immediately applicable to our study. We next propose a new fast random heuristic method for computing the strong rainbow connection numbers of general graphs. 

\begin{Ualgorithm}[t]
\DontPrintSemicolon
\KwData{A simple connected graph $G$, and an integer $maxIter$.}
\KwResult{$f$ (a valid strong rainbow coloring of $G$) and $best$ (the number of colors used by $f$).}
$best\gets m+1$\label{line:2}\;
\ForEach{$i = 1, \dots, maxIter$\label{line:3}}{
    $k\gets0$\label{line:1}\;
    Fix a random shortest $(u,v)$-path $P_{uv}$ for all $u\neq{v}$ in $V(G)$\label{line:fix-paths}\;
    $free\gets E(G)$\;
    $fixed\gets\emptyset$\;
    \While{$free \neq \emptyset$}{
        randomly select an edge $e$ in $free$\;
        $free \gets free \backslash \{e\}$\;
        $K' \gets \{1, \dots, k\}$\;
        \ForEach{$u \neq v \in V(G)$}{
            \If{$K' = \emptyset$}{
                 \Goto line 17\;
            }
            \If{$e \in E(P_{uv})$}{
                \ForEach{$e' \in E(P_{uv}) \cap fixed$}{
                    $K' \gets K' \backslash \{f(e')\}$\;
                }
            }
        }
        \eIf{$K' \neq \emptyset$}{
            $f(e) \gets$ randomly chosen element of $K'$\;
        }{
            $k \gets k+1$\;
            \If{$k\geq best$}{
                \Goto line 25\;
            }
            $f(e) = k$\;
        }
        $fixed \gets fixed \cup \{e\}$\;
    }
    $best = \min(best, k)$\;
}
\Return{$f$, $best$}
\caption{Fast Random Strong Rainbow Coloring Heuristic.}
\label{alg:1} 
\end{Ualgorithm}

The heuristic, presented in \autoref{alg:1}, begins by randomly selecting a single shortest $(u,v)$-path $P_{uv}\in\mathcal{P}_{uv}$ for all $u\neq{v}\in{V(G)}$
(in the coloring returned upon termination, each of the selected paths $P_{uv}$ will be a strong rainbow path).
Throughout execution, the heuristic counts the number of colors $k$ used in the strong rainbow coloring (initially $k=0$). The heuristic first randomly selects an uncolored edge $e\in{E(G)}$. It then checks whether edge $e$ can be colored using a color that has already been used previously (i.e.~one of the colors $\{1,\dots,k\}$), without causing any selected path to use the same color twice. If so, then $e$ is colored with such a color. Otherwise, $e$ is colored with a new color, and $k$ is increased by one. A new uncolored edge $e\in{E(G)}$ is selected, and the process repeats until all edges in $E(G)$ are colored.
In order to produce better solutions, the
entire randomized process can be repeated multiple times (controlled by the parameter \texttt{maxIter}).

%
This heuristic requires the selection of a path $P_{v_1v_2}\in\mathcal{P}_{v_1v_2}$ for all $v_1\neq{v_2}\in{V(G)}$. We now describe a method to uniformly sample elements from $\mathcal{P}_{v_1v_2}$. Let $r_{v_1v_2}=|\mathcal{P}_{v_1v_2}|$ equal the number of distinct shortest $(v_1,v_2)$-paths in $G$---in \autoref{subsec:imp}, we explained how these values can be quickly determined in $O(n^2+nm)$ time. 
%
To sample a shortest $(v_1,v_2)$-path, we begin at vertex $v_2$ in the directed graph $D_{v_1}$, and move back to vertex $v_1$ as follows. Set $u_0=v_2$ to be the initial vertex in the path. For each $i=1,\dots,d(v_1,v_2)$, the $i^\mathrm{th}$ vertex $u_i$ in the path is sampled from the set $N^+(u_{i-1})$, where vertex $u\in{N^+(u_{i-1})}$ is chosen with probability $r_{v_1u}/\sum_{u'\in{N^+(u_{i-1})}}r_{v_1 u'}$. An example of the labels $r_{uv}$ is shown in \autoref{fig:dag-example}.


Though not included in the provided pseudocode for readability, it is possible to improve the performance of \autoref{alg:1} by initially coloring an edge set $U\subseteq V(H) = E(G)$ which is a clique in $H$ with distinct colors and adjusting the initial values of $fixed, free, k,$ and $f$ accordingly. Doing so creates a larger initial set of colors, providing additional flexibility in color choice and modestly improving the quality of solutions. Because the maximum cardinality clique in $H$ is already computed during the model preprocessing outlined in \autoref{subsec:ip-imp}, it can easily be reused here. Alternatively if a worst-case polynomial time algorithm is desired, the heuristic could similarly be improved by initially coloring one of the fixed $\diam(G)$ length paths chosen in line \ref{line:fix-paths} instead.

It is also possible to extend the proposed heuristic to the standard rainbow connection problem. While any coloring generated by this heuristic must strongly rainbow connect $G$, and thus also be a valid rainbow coloring of $G$, additional rainbow connecting edge colorings may be generated by fixing randomly chosen $(u,v)$-paths for each $u \neq v \in V(G)$ rather than fixing randomly chosen $(u,v)$-\textit{shortest}-paths.

\subsection{The Bottom-Up Approach}\label{subsec:bottom-up}

We now introduce a technique for computing $src(G)$ which takes advantage of the fact that the lower bound $\omega'(G)$ presented in \autoref{sec:lb} is often very strong in practice and can be computed quite quickly despite its theoretical hardness. We refer to this method as the bottom-up approach. The method proceeds as follows: we first compute the lower bound $\ell{b}=\max(\diam(G),\omega'(G))$ for $src(G)$. Next, we solve \refp{ip:master} with $|K|=\ell{b}$ (i.e.~with $\ell{b}$ available colors). If the resulting model is infeasible, we increase $|K|$ by 1, and repeat this process until we obtain a feasible instance. The first value of $|K|$ for which \refp{ip:master} is feasible equals $src(G)$. This method clearly terminates in finite time. 

The success of this method in practice relies on two factors. First, because \refp{ip:master} has $O(|\mathcal{P}| + (m+1) |K| )$ variables, $O((|\mathcal{P}| + m + 1) |K|)$ constraints, and $|\mathcal{P}|$ is relatively large, model size grows quickly with $|K|$. To ensure feasibility of \refp{ip:master}, we require $K_0=|K|\geq src(G)$, at best choosing $K_0$ to be equal to the best known upper bound on $src(G)$. In contrast, the bottom-up approach initializes $K_0$ to be equal to the best known lower bound, often formulating and solving a considerably smaller model. Second, the bound $\max(\diam(G),\omega'(G))$ is often very close to $src(G)$, so the bottom-up approach requires very few iterations to converge. 
In \autoref{sec:computation} we demonstrate that this method is very effective in practice. 

We note that the above technique does not require the exact computation of $\omega'(G)$---any lower bound on $src(G)$ is sufficient. Since $\omega'(G)$ often provides a particularly good bound on $src(G)$ and can be quickly computed, we choose to utilize it here. Moreover, while larger cliques in $H$ produce tighter bounds on $\omega'(G)$ and thus on $src(G)$, a clique of any size can in fact be used. In particular, any existing max clique heuristic (e.g.~\cite{benlic2013breakout, ILS-VND, wu2012multi}) could also be used to find a large clique in $H$. Due to the relatively little computational time needed to solve max clique exactly for the problem instances we consider, in \autoref{sec:computation}, we chose to compute $\omega'(G)$ exactly.

\subsection{Model Extensions for Related Problems}\label{subsec:ip-ext}

In this section, we describe how the integer program \refp{ip:master} can be extended to solve several popular variants of the rainbow connection problem which have been introduced in the literature: the rainbow $k$-connection problem~\cite{chartrand2009}, rainbow vertex connection problem~\cite{rvc}, the strong rainbow vertex connection problem~\cite{srvc}, and the very strong rainbow connection problem~\cite{vsrc} (see Li et al.~\cite{li2013a} for a review of several of these variants and many of their known results).

Among these variants, perhaps most similar to the strong rainbow connection problem is the very strong rainbow connection problem. In this problem we also consider an edge coloring $f:E(G) \to \{1, \dots, k'\}$ with $k' \in \mathbb{N}$, but for each pair of distinct vertices $u,v$, rather than requiring that at least one shortest $(u,v)$-path is rainbow with respect to $f$, we now require that \textit{every} $(u,v)$-shortest-path is rainbow with respect to $f$. The smallest $k' \in \mathbb{N}$ for which there exists an $f$ that strongly rainbow connects $G$ is known as the {\em very strong rainbow connection number} of $G$. \refp{ip:master} can be extended to this problem by simply requiring that $y_P = 1$ for all $P$ in $\mathcal{P}$. Thus all but a polynomial number of variables ($x_{e,k}$, $z_k, \: \forall e \in E(G), k \in K$) can be eliminated, constraints \refp{ip:master-sum-y} are necessarily satisfied and can therefore be eliminated, and constraints \refp{ip:master-big-m} reduce to set packing constraints.

The rainbow $k$-connection problem extends the standard rainbow connection problem by requiring that each pair of vertices are connected by at least $k$ internally vertex disjoint rainbow paths, with respect to an edge coloring $f:E(G) \to \{1, \dots, k'\}$. The {\em rainbow $k$-connectivity number} of $G$ is then the smallest $k'$ for which there exists an edge coloring $f$ that rainbow $k$-connects $G$. As previously indicated, \refp{ip:master} exactly solves the rainbow connection problem when the set $\mathcal{P}$ is extended to the set of all paths in $G$. \refp{ip:master} can be further extended to the rainbow $k$-connection problem by expanding each $\mathcal{P}_{uv}$ to be the set of all $(u,v)$-paths in $G$, replacing the right hand side of constraints \refp{ip:master-sum-y} with $k$, and adding additional constraints of the form:
$$ y_{P} + y_{P'} \leq 1, \qquad \qquad \forall u \neq v \in V(G), \forall P \neq P' \in \mathcal{P}_{uv}: V(P) \cap V(P') \backslash \{u,v\} \neq \emptyset. $$

Finally, the rainbow vertex connection problem and the strong rainbow vertex connection problems explore the notion of rainbow coloring in the case that vertices are colored rather than edges. In these settings we consider a vertex coloring $f: V(G) \to \{1, \dots, k'\}$ with $k' \in \mathbb{N}$. A pair of vertices $v_1 \neq v_2 \in V(G)$ are said to be {\em rainbow vertex connected} if there exists a $(v_1,v_2)$-path $P$ such that for any $u_1 \neq u_2 \in V(P)$, $f(v_1) \neq f(v_2)$. If every pair of vertices in $G$ is rainbow vertex connected, then $f$ is said to {\em rainbow vertex connect} $G$. If every pair of vertices in $G$ is rainbow vertex connected via a shortest path, then $f$ is also said to {\em strongly rainbow vertex connect} $G$. \refp{ip:master} can be modified for the strong rainbow vertex connection problem by replacing the $x_{e,k}$ variables with $x_{v,k}$ variables to indicate a vertex coloring rather than an edge coloring. The constraints \refp{ip:master-sum-one}, \refp{ip:master-big-m}, and \refp{ip:master-vub} can all be adjusted analogously to apply to vertex colorings rather than edge colorings. While this formulation solves the strong rainbow vertex connection problem, the standard rainbow vertex connection problem can also be solved by extending the set $\mathcal{P}$ to be the set of all paths in $G$, as in the case of the standard rainbow connection problem.

\section{Computational Results}\label{sec:computation}

In this section, we conduct a series of computational experiments to assess the performance of the methods we have introduced for calculating $src(G)$:
\begin{enumerate}
\item \textbf{Na\"ive Model:} In this test we directly solve the integer program \refp{ip:master}. We do not use any of the computational enhancements introduced in \autoref{subsec:ip-imp} (no variable fixing, constraint elimination, etc.). This method uses the heuristic solution (\autoref{subsec:heuristic}) as the value for $K_0$.
\item \textbf{Enhanced Model:} In this test we directly solve \refp{ip:master}, this time using the computational enhancements of \autoref{subsec:ip-imp}. The initial variable fixing was created with the maximum clique in $H$, computed using the Cliquer code~\citep{cliquer}. This method also uses the heuristic solution (\autoref{subsec:heuristic}) as the value for $K_0$.
\item \textbf{Bottom-Up:} This test computes $src(G)$ using the bottom-up approach introduced in \autoref{subsec:bottom-up}. The initial value for $K_0$ is equal to $\max(\diam(G),\cliquelb(G))$, where $\cliquelb(G)$ is again computed using Cliquer. This method also uses the computational enhancements of \autoref{subsec:ip-imp}. Unlike the previous two methods, the bottom-up approach does not benefit from an initial feasible solution or upper bound, and thus the heuristic in \autoref{subsec:heuristic} is not used.
\end{enumerate}

To test these methods, we consider the following three categories of graphs:
\begin{enumerate}
\item \textbf{Social and Infrastructure Graphs:} We consider several graphs which represent real-world social and communication networks\footnote{Publicly available in the Koblenz Network Collection~\cite{konect} and Network Repository~\cite{nr}.} \cite{surfers, jazz, lesmis, dolphins, football, adjnouns, radoslaw, karate}, and a set of standard IEEE graphs used to represent infrastructure networks\footnote{Publicly available at the Power Systems Test Case Achieve: \url{https://labs.ece.uw.edu/pstca/}.}. These graphs are similar to the communication networks which motivate the applications of the strong rainbow connection problem discussed in \autoref{sec:introduction}.
\item \textbf{Sparse Random Graphs:} We also consider two popular random graph models: the Erd\H{o}s-R\'enyi random graph model \cite{ER-graphs} and Watts-Strogatz random graph model \cite{watts-strogatz}.
The Watts-Strogatz model is particularly relevant to the strong rainbow connection problem as these graphs exhibit the so called `small world' property often observed in social graphs -- shortest paths are often quite short in length despite potentially large numbers of vertices. 
\item \textbf{Dense Random Bipartite Graphs:} As demonstrated in \autoref{tab:stats} in \autoref{subsec:experiments}, it is often the case that the lower bound given by $\max(diam(G),\cliquelb(G))$ is tight in the first two categories of graphs we investigate. Consequently, the bottom-up approach of \autoref{subsec:bottom-up} generally outperforms the other methods we consider. In order to assess potential limitations of the bottom-up approach when this lower bound is not tight, we also examine dense random bipartite graphs. 
Although these graphs do not resemble the real-world graphs which motivate the strong rainbow connection problem, this choice is motivated by \autoref{prop:arb-far}, which shows that the bound $\max(\diam(G),\cliquelb(G))$ can be arbitrarily poor in complete bipartite graphs of the form $K_{2,n}$.
\end{enumerate}

\subsection{Implementation Details}\label{subsec:imp-details}

Each test was executed on the same hardware, a laptop with 32 GB of RAM running Ubuntu version 19.10 with a 2.60 GHz i7 processor. All tests were performed using a 1 hour timeout limit. 
We used Python 3.7.4 and the Networkx package \cite{networkx} to implement the heuristic method and various preprocessing graph algorithms described in \autoref{sec:ip-land}. 
Gurobi \cite{gurobi}, version 9.0.2, was used to solve integer programs. 
Our code and the test instances we consider are publicly available at \url{github.com/dtmildebrath/rainbow}.

To compute cliques,
we considered the publicly available max (weight) clique solver Cliquer~\cite{cliquer} as well as a max weight independent set hybrid iterated local search heuristic (ILS-VND) proposed by Nogueira et al.~\cite{ILS-VND}. 
As identifying max weight cliques is $\mathcal{NP}$-hard in general, we expected that exactly computing maximum cliques would be less computationally effective than using a heuristic method such as ILS-VND. In our experiments however, we found that due to the sizes of the subproblems which we encounter in our tests and the low edge densities of the corresponding auxiliary graphs, very little time is needed to solve these subproblems exactly. As such, the reported results use Cliquer to 
compute max cliques,
and note that in the vast majority of our test instances less than 1\% of the total runtime was spent on this task.

Some of the social network graphs we considered as test cases included directed edges, loops, or parallel edges. Since these were often used to indicate repeated communications, such as emails sent from one party and received by another party or messages sent from a party to back to that same party, these graphs were modified so that loops were removed, parallel edges were reduced to a single edge, and directed edges were made undirected. In the case that a graph had multiple connected components, only the component with the largest number of vertices is considered. 

We also consider three random graph models.
The first is the Erd\H{o}s-R\'enyi random graph model, $G_{n,p}$, where the generated graph has $n$ vertices, and each pair of vertices are adjacent with probability $p$. The test instance labeled `ER\_$n_0$\_$p_0$\_$i$' is then the $i^\text{th}$ random graph generated with $n = n_0$ and $p = p_0/100$. The second model is the Watts-Strogatz random graph model with parameters $n$, $k$, $p$. In this model, $n$ indicates a number of nodes initially organized into a cycle. These nodes are then connected to their $k$ nearest neighbors and finally each edge is removed and replaced by an edge not in the graph, with probability $p$. These test instances are labeled such that `WS\_$n_0$\_$k_0$\_$p$\_$i$' is the $i^\text{th}$ random graph generated with $n=n_0$, $k=k_0$, and $p=p_0/100$. The third model is the bipartite analogue of the Erd\H{o}s-R\'enyi random graph model, $G_{n_1, n_2, p}$, where the generated graph is a bipartite graph with parts containing $n_1$ and $n_2$ vertices, and each pair of vertices which are not both in the same part are adjacent with probability $p$. Test instance `$BER$\_$n_a$\_$n_b$\_$p_0$\_$i$' is the $i^\text{th}$ random bipartite graph generated with $n_1=n_a, n_2=n_b$, and probability $p=p_0/100$.

Aside from the default parameters provided by the Gurobi solver, 
the only tunable parameter which appears in any of our methods is $maxIter$, the number of iterations for which the random heuristic introduced in \autoref{subsec:heuristic} is run. Despite its simplicity, we found that when combined with the lower bound $\cliquelb(G)$ the random heuristic often provides a surprisingly tight optimality gap with $maxIter = \lceil n /5\rceil$. As expected, runtime performance benefits diminish as $maxIter$ is increased and more search time needed for small improvements in solutions. In \autoref{tab:times} we provide the solutions returned by the heuristic with $maxIter= \lceil n /5\rceil$ and its total run time. 
For each test instance the heuristic was run once; the resulting solution was then used 
for all subsequent experiments,
and the time needed for the heuristic was added to the overall runtimes reported in \autoref{tab:times} as appropriate.

\subsection{Results and Discussion}\label{subsec:experiments}


\begin{table}\scriptsize\centering
\makebox[\textwidth][c]{ 
\begin{tabular}{c|cccc....}
\hline
Instance & $n$ & $src(G)$ & Init.~LB & Heur.~UB & \mc{Heur.~Time}  & \mc{Naive} & \mc{Enhanced} & \mc{Bott.~up} \\
\hline
ER\_80\_3.2\_0    & 80  & 19  & 19 & 27  & 0.521   & \mc{[14,\;21]} & 8.621            & \boldc{6}{053}   \\
ER\_80\_3.2\_1    & 80  & 16  & 16 & 27  & 0.525   & 1512.083       & 25.884           & \boldc{9}{226}   \\
ER\_80\_3.2\_2    & 80  & 21 & 21 & 26  & 0.510   & \mc{[16,\;21]} & 3.306            & \boldc{1}{740}   \\
ER\_80\_3.2\_3    & 80  & 24 & 24 & 30  & 0.474   & \mc{[20,\;24]} & 3.180            & \boldc{1}{584}   \\
ER\_80\_3.2\_4    & 80  & 25 & 25 & 29  & 0.489   & \mc{[24,\;25]} & 2.587            & \boldc{1}{519}   \\
\hline
ER\_80\_6\_0      & 80  & 8  & 8 & 18  & 1.310   & 2319.224       & 1887.042         & \boldc{210}{050} \\
ER\_80\_6\_1      & 80  & 10 & 10 & 21  & 1.013   & 1287.704       & 888.659          & \boldc{79}{950}  \\
ER\_80\_6\_2      & 80  & 8 & 8  & 18  & 1.444   & 1688.089       & 636.261          & \boldc{92}{896}  \\
ER\_80\_6\_3      & 80  & 9 & 9  & 19  & 1.361   & 871.583        & 141.416          & \boldc{2}{230}   \\
ER\_80\_6\_4      & 80  & 8 & 8  & 17  & 1.689   & 670.346        & 510.330          & \boldc{69}{112}  \\
\hline
ER\_100\_4\_0     & 100 & 13 & 13 & 26  & 2.520   & \mc{[13,\;14]} & 2986.063         & \boldc{317}{650} \\
ER\_100\_4\_1     & 100 & 11 & 11 & 24  & 3.004   & \mc{[11,\;13]} & \boldc{3295}{713} & \mc{[11,\;-]}    \\ 
ER\_100\_4\_2     & 100 & - & 10  & 25  & 3.186   & \mc{[10,\;16]} & \mc{[10,\;16]}   & \mc{[10,\;-]}    \\
ER\_100\_4\_3     & 100 & - & 10  & 24  & 2.689   & \mc{[10,\;14]} & \mc{[10,\;16]}   & \mc{[11,\;-]}    \\
ER\_100\_4\_4     & 100 & - & 10  & 24  & 2.959   & \mc{[10,\;12]} & \mc{[10,\;12]}   & \mc{[10,\;-]}    \\
\hline
ER\_100\_8\_0     & 100 & 7 & 7  & 16  & 6.905   & 1827.875       & \mc{[7,\;8]}     & \boldc{631}{530} \\
ER\_100\_8\_1     & 100 & 7 & 7  & 16  & 8.003   & 747.139        & 663.521          & \boldc{4}{161}   \\
ER\_100\_8\_2     & 100 & 6 & 6  & 17  & 6.647   & \mc{[6,\;7]}   & \mc{[6,\;7]}     & \boldc{1072}{940}\\
ER\_100\_8\_3     & 100 & 6 & 6  & 15  & 7.576   & 3108.940       & 3361.426         & \boldc{4}{582}   \\
ER\_100\_8\_4     & 100 & 6 & 6  & 16  & 7.530   & 3160.943       & \mc{[6,\;8]}     & \boldc{4}{178}   \\
\hline
WS\_100\_10\_1\_0 & 100 & 8 & 8  & 22  & 10.090  & \mc{[8,\;10]}  & \mc{[8,\;10]}    & \boldc{17}{050}  \\
WS\_100\_10\_1\_1 & 100 & 9 & 9 & 24  & 14.139  & \mc{[8,\;17]}  & \mc{[8,\;24]}    & \boldc{1612}{366}\\
WS\_100\_10\_1\_2 & 100 & - & 8  & 25  & 13.031  & \mc{[8,\;13]}  & \mc{[8,\;25]}    & \mc{[8,\;-]}     \\
WS\_100\_10\_1\_3 & 100 & 9 & 9  & 24  & 12.147  & \mc{[9,\;19]}  & \mc{[9,\;10]}    & \boldc{34}{599}  \\
WS\_100\_10\_1\_4 & 100 & - & 7  & 22  & 9.911   & \mc{[7,\;11]}  & \mc{[7,\;11]}    & \mc{[7,\;-]}     \\
\hline
WS\_100\_20\_1\_0 & 100 & 4 & 4  & 13  & 16.986  & 673.842        & 475.729          & \boldc{9}{915}   \\
WS\_100\_20\_1\_1 & 100 & 4 & 4  & 14  & 19.463  & 3068.778       & 879.595          & \boldc{11}{697}  \\
WS\_100\_20\_1\_2 & 100 & 5 & 5  & 14  & 18.251  & \mc{[3,\;5]}   & \mc{[4,\;5]}     & \boldc{13}{542}  \\
WS\_100\_20\_1\_3 & 100 & 5 & 5  & 14  & 16.293  & \mc{[4,\;5]}   & 1773.939         & \boldc{12}{433}  \\
WS\_100\_20\_1\_4 & 100 & 5 & 5  & 14  & 17.357  & \mc{[4,\;6]}   & \mc{[4,\;5]}     & \boldc{14}{775}  \\
\hline
WS\_150\_30\_1\_0 & 150 & 4 & 4  & 14  & 123.940 & \mc{[3,\;9]}   & \mc{[3,\;5]}     & \boldc{43}{778}  \\
WS\_150\_30\_1\_1 & 150 & 4 & 4  & 15  & 129.283 & \mc{[3,\;7]}   & \mc{[4,\;15]}    & \boldc{49}{698}  \\
WS\_150\_30\_1\_2 & 150 & 4 & 4  & 14  & 120.381 & \mc{[3,\;9]}   & \mc{[4,\;14]}    & \boldc{35}{356}  \\
WS\_150\_30\_1\_3 & 150 & 4 & 4  & 15  & 122.179 & \mc{[3,\;8]}   & \mc{[4,\;15]}    & \boldc{46}{268}  \\
WS\_150\_30\_1\_4 & 150 & 4 & 4  & 15  & 132.295 & \mc{[4,\;12]}  & \mc{[4,\;15]}    & \boldc{37}{361}  \\
\hline
ieee30            & 30  & 10 & 10  & 12  & 0.019   & 0.252          & 0.106            & \boldc{0}{090}   \\
ieee57            & 57  & 14 & 14 & 22  & 0.115   & 54.571         & 4.893            & \boldc{1}{470}   \\
ieee118           & 118 & 25 & 25 & 42  & 2.705   & \mc{[25,\;32]} & 269.386          & \boldc{225}{037} \\
ieee300           & 300 & (107) & 106 & 125 & 127.574 & \mc{[-,\;125]} & \mc{[106,\;108]} & \mc{[107,\;-]}   \\
\hline
karate            & 34  & 6  & 6 & 11  & 0.069   & 3.001          & 0.798            & \boldc{0}{282}   \\
surfers           & 43  & 3  & 3 & 6   & 0.401   & 1.470          & 1.145            & \boldc{0}{394}   \\
dolphins          & 62  & 10 & 10 & 17  & 0.524   & \mc{[9,\;10]}  & 6.444            & \boldc{1}{265}   \\
lesmis            & 77  & 21 & 21 & 27  & 1.741   & 138.235        & 5.694            & \boldc{3}{150}   \\
adjnoun           & 112 & 11 & 11 & 25  & 8.723   & \mc{[11,\;12]} & 1000.355         & \boldc{7}{827}   \\
football          & 115 & (5)& 5 & 16  & 14.334  & \mc{[5,\;8]}   & \mc{[5,\;8]}     & \mc{[5,\;-]}     \\
enron-emails      & 143 & 10 & 10 & 27  & 26.735  & \mc{[9,\;18]}   & \mc{[10,17\;]}     & \boldc{16}{389}  \\
rado-emails       & 167 & 24 & 24 & 24  & 214.839 & \mc{[4,\;24]}  & 309.058          & \boldc{58}{173}  \\
jazz-musicians    & 198 & -  & 6 & 22  & 416.739 & \mc{[5,\;22]}  & \mc{[6,\;22]}    & \mc{[6,\;-]}     \\
ca-netscience     & 379 & 89 & 89 & 118 & 1213.924 & \mc{[-,\;-]}  & \mc{[-,\;-]}     & \boldc{1475}{184} \\
\hline
\end{tabular}
} 
\caption{\footnotesize Performance comparison of proposed methods. The lower bound $\max(\diam(G), \cliquelb(G))$ is indicated as Init. LB and the upper bound obtained by \autoref{alg:1} is indicated as Heur. UB. A 1 hour time limit was used. If a method timed out before a solution was found, the best lower and upper bounds obtained within the timeout period (if any) are given in brackets: $[\mathrm{lb},\mathrm{ub}]$. The bottom-up approach does not produce an upper bound, and thus ub$=$``-'' for these tests. The fastest method for each instance is indicated in bold. $src$ values reported in parentheses were obtained using extended runtimes.}
\label{tab:times}
\end{table}

\begin{table}\scriptsize\centering
\begin{tabular}{c|rrccccrrr}
    \hline
    Instance & $n$ & $m$ & $src(G)$ & $\cliquelb(G)$ & $\diam(G)$ & $dens(H)$ & $|\mathcal{P}|$ & $|\mathcal{P}|$ rem. & \% rem. \\ \hline
    ER\_80\_3.2\_0      & 80  & 124  & 19  & 19  & 9  & 3.92  & 5655   & 2384   & 42.2 \\
    ER\_80\_3.2\_1       & 80  & 121  & 16  & 16  & 10 & 3.83  & 5276   & 2175   & 41.2 \\
    ER\_80\_3.2\_2       & 80  & 124  & 21  & 21  & 10 & 3.92  & 5708   & 2308   & 40.4 \\
    ER\_80\_3.2\_3       & 80  & 115  & 24  & 24  & 9  & 3.64  & 4906   & 1910   & 38.9 \\
    ER\_80\_3.2\_4       & 80  & 115  & 25  & 25  & 11 & 3.64  & 4621   & 1901   & 41.1 \\
    \hline
    ER\_80\_6\_0         & 80  & 189  & 8   & 8   & 6  & 5.98  & 6670   & 5180   & 77.7 \\
    ER\_80\_6\_1         & 80  & 163  & 10  & 10  & 7  & 5.16  & 6327   & 4132   & 65.3 \\
    ER\_80\_6\_2         & 80  & 199  & 8   & 8   & 6  & 6.30  & 6990   & 5717   & 81.8 \\
    ER\_80\_6\_3         & 80  & 191  & 9   & 9   & 5  & 6.04  & 6524   & 5175   & 79.3 \\
    ER\_80\_6\_4         & 80  & 211  & 8   & 8   & 5  & 6.68  & 5848   & 5848   & 100.0 \\
    \hline
    ER\_100\_4\_0        & 100 & 185  & 13  & 13  & 7  & 3.74  & 9221   & 5725   & 62.1 \\
    ER\_100\_4\_1        & 100 & 206  & 11  & 11  & 7  & 4.16  & 9571   & 6500   & 67.9 \\
    ER\_100\_4\_2        & 100 & 208  & -   & 10  & 8  & 4.20  & 10293  & 6844   & 66.5 \\
    ER\_100\_4\_3        & 100 & 200  & -   & 10  & 6  & 4.04  & 9730   & 6005   & 61.7 \\
    ER\_100\_4\_4        & 100 & 218  & -   & 10  & 6  & 4.40  & 10580  & 7580   & 71.6 \\
    \hline
    ER\_100\_8\_0        & 100 & 378  & 7   & 7   & 4  & 7.64  & 13795  & 13102  & 95.0 \\
    ER\_100\_8\_1        & 100 & 402  & 7   & 7   & 4  & 8.12  & 14879  & 14270  & 95.9 \\
    ER\_100\_8\_2        & 100 & 378  & 6   & 6   & 4  & 7.64  & 13503  & 12711  & 94.1 \\
    ER\_100\_8\_3        & 100 & 409  & 6   & 6   & 4  & 8.26  & 15511  & 14896  & 96.0 \\
    ER\_100\_8\_4        & 100 & 408  & 6   & 6   & 4  & 8.24  & 15119  & 14531  & 96.1 \\
    \hline
    WS\_100\_10\_1\_0 & 100 & 500  &  8  & 8   & 8  & 10.10 & 51208  & 47253  & 92.3 \\
    WS\_100\_10\_1\_1 & 100 & 500  &  9  & 9   & 8  & 10.10 & 97045  & 93088  & 95.9 \\
    WS\_100\_10\_1\_2 & 100 & 500  &  -  & 8   & 8  & 10.10 & 92757  & 88095  & 95.0 \\
    WS\_100\_10\_1\_3 & 100 & 500  &  9  & 9   & 8  & 10.10 & 80985  & 77201  & 95.3 \\
    WS\_100\_10\_1\_4 & 100 & 500  &  -  & 7   & 7  & 10.10 & 40077  & 37174  & 92.8 \\
    \hline
    WS\_100\_20\_1\_0 & 100 & 1000 & 4   & 4   & 4  & 20.20 & 45194  & 44518  & 98.5 \\
    WS\_100\_20\_1\_1 & 100 & 1000 & 4   & 4   & 4  & 20.20 & 52548  & 51895  & 98.8 \\
    WS\_100\_20\_1\_2 & 100 & 1000 & 5   & 3   & 5  & 20.20 & 59160  & 58435  & 98.8 \\
    WS\_100\_20\_1\_3 & 100 & 1000 & 5   & 4   & 5  & 20.20 & 57989  & 57291  & 98.8 \\
    WS\_100\_20\_1\_4 & 100 & 1000 & 5   & 4   & 5  & 20.20 & 67365  & 66527  & 98.8 \\
    \hline 
    WS\_150\_30\_1\_0 & 150 & 2250 &  4  & 3   & 4  & 20.13 & 179275 & 178655 & 99.7 \\
    WS\_150\_30\_1\_1 & 150 & 2250 &  4  & 4   & 4  & 20.13 & 197491 & 196563 & 99.5 \\
    WS\_150\_30\_1\_2 & 150 & 2250 &  4  & 4   & 4  & 20.13 & 147746 & 147186 & 99.6 \\
    WS\_150\_30\_1\_3 & 150 & 2250 &  4  & 4   & 4  & 20.13 & 186169 & 185445 & 99.6 \\
    WS\_150\_30\_1\_4 & 150 & 2250 &  4  & 4   & 4  & 20.13 & 154201 & 153663 & 99.7 \\
    \hline
    ieee30                         & 30  & 41   & 10  & 10  & 6  & 9.43  & 484    & 201    & 41.5 \\
    ieee57                         & 57  & 78   & 14  & 14  & 12 & 4.89  & 2147   & 699    & 32.6 \\
    ieee118                        & 118 & 179  & 25  & 25  & 14 & 2.59  & 15600  & 6237   & 40.0 \\
    ieee300                        & 300 & 409  & 107 & 106 & 24 & 0.91  & 125782 & 36954  & 29.4 \\
    \hline
    karate                         & 34  & 78   & 6   & 6   & 5  & 13.90 & 1478   & 1315   & 89.0 \\
    surfers                        & 43  & 336  & 3   & 3   & 3  & 37.21 & 3470   & 3470   & 100.0 \\
    dolphins                       & 62  & 159  & 10  & 10  & 8  & 8.41  & 5513   & 3339   & 60.6 \\
    lesmis                         & 77  & 254  & 21  & 21  & 5  & 8.68  & 6581   & 5381   & 81.8 \\
    adjnoun                        & 112 & 425  & 11  & 11  & 5  & 6.84  & 23535  & 19083  & 81.1 \\
    football                       & 115 & 613  & 5   & 5   & 4  & 9.35  & 22150  & 21114  & 95.3 \\
    enron-emails                   & 143 & 623  & 10  & 10  & 8  & 6.14  & 42979  & 36048  & 83.9 \\
    rado-emails                    & 167 & 3251 & 24  & 24  & 5  & 23.45 & 137009 & 108349 & 79.1 \\
    jazz-musicians                 & 198 & 2742 & -   & 6   & 6  & 14.06 & 186618 & 172595 & 92.5 \\
    ca-netscience                  & 379 & 914  & 89  & 89  & 17 & 1.28  & 189790 & 131535 & 69.3 \\
    \hline
\end{tabular}
\caption{\footnotesize Additional information for the graphs considered in \autoref{tab:times}. For each graph $G$ (from left to right) the order, size, $src(G)$, lower bound $\cliquelb(G)$, $\diam(G)$, edge density of $H(G)$ ($dens(H) := |E(H)|/\binom{|V(H)|}{2}$), number of shortest paths in $G$, number of shortest paths in $G$ remaining after variable elimination, and the percentage of shortest paths in $G$ remaining after variable elimination are indicated. }
\label{tab:stats}
\end{table}

\begin{table}\scriptsize\centering
\makebox[\textwidth][c]{ 
\begin{tabular}{c|cccc....}
\hline
Instance & $n$ & $src(G)$ & Init.~LB & Heur.~UB & \mc{Heur.~Time}  & \mc{Naive} & \mc{Enhanced} & \mc{Bott.~up} \\
\hline
BER\_2\_25\_95\_0    & 27  & 6 & 3 & 9  & 0.028   & \mc{[5,\;6]} & \boldc{37}{286}            & 581.085          \\
BER\_2\_25\_95\_1    & 27  & 5 & 2 & 8  & 0.029   & 77.796       & \boldc{23}{364}            & 40.104           \\
BER\_2\_25\_95\_2    & 27  & 8 & 6 & 10 & 0.027   & \mc{[7,\;8]} & \boldc{2788}{336}          & \mc{[7,\;-]}      \\
BER\_2\_25\_95\_3    & 27  & 7 & 4 & 9  & 0.027   & \mc{[6,\;7]} & \boldc{44}{861}            & \mc{[6,\;-]}     \\
BER\_2\_25\_95\_4    & 27  & 7 & 4 & 10 & 0.027   & \mc{[6,\;7]} & \mc{[6,\;7]}               & \boldc{114}{553} \\
\hline
BER\_2\_25\_90\_0   & 27  & 7 & 4 & 9 & 0.025   & \mc{[6,\;7]}  & 1873.748         & \boldc{61}{777}    \\
BER\_2\_25\_90\_1   & 27  & 7 & 4 & 9 & 0.025   & \mc{[6,\;7]}  & \boldc{304}{780} & 935.404    \\
BER\_2\_25\_90\_2   & 27  & 6 & 4 & 9 & 0.027   & 111.781       & 2781.163         & \boldc{104}{679}    \\
BER\_2\_25\_90\_3   & 27  & 6 & 4 & 9 & 0.026   & \mc{[5,\;6]}  & \boldc{124}{957} & 196.738    \\
BER\_2\_25\_90\_4   & 27  & 7 & 4 & 9 & 0.025   & \mc{[6,\;7]}  & \boldc{208}{203} & 379.667    \\
\hline
BER\_2\_25\_85\_0   & 27  & 8 & 6 & 9  & 0.027   & 787.454      & 1.334                      & \boldc{0}{606}   \\
BER\_2\_25\_85\_1   & 27  & 9 & 8 & 10 & 0.027   & 14.375       & \boldc{0}{770}             & 0.824            \\
BER\_2\_25\_85\_2   & 27  & 8 & 6 & 10 & 0.027   & \mc{[7,\;8]} & 11.767                     & \boldc{6}{340}   \\
BER\_2\_25\_85\_3   & 27  & 8 & 6 & 10 & 0.027   & \mc{[7,\;8]} & \boldc{145}{622}           & 214.679          \\
BER\_2\_25\_85\_4   & 27  & 8 & 7 & 9  & 0.027   & 23.749       & 1.022                      & \boldc{0}{695}  \\
\hline
BER\_2\_25\_80\_0   & 27  & 9 & 8 & 12 & 0.024   & 25.743      & 0.764         & \boldc{0}{646}    \\
BER\_2\_25\_80\_1   & 27  & 9 & 8 & 13 & 0.025   & 422.235     & 0.817         & \boldc{0}{734}    \\
BER\_2\_25\_80\_2   & 27  & 9 & 9 & 11 & 0.039   & 15.630      & 0.210         & \boldc{0}{160}    \\
BER\_2\_25\_80\_3   & 27  & 8 & 6 & 10 & 0.026   & 2383.839    & 20.677        & \boldc{7}{788}    \\
BER\_2\_25\_80\_4   & 27  & 8 & 7 & 11 & 0.025   & 25.990      & 0.936         & \boldc{0}{599}    \\
\hline
\end{tabular}
} 
\caption{\footnotesize Performance comparison of proposed methods. The lower bound $\max(\diam(G), \cliquelb(G))$ is indicated as Init. LB and the upper bound obtained by \autoref{alg:1} is indicated as Heur. UB. A 1 hour time limit was used. If a method timed out before a solution was found, the best lower and upper bounds obtained within the timeout period (if any) are given in brackets: $[\mathrm{lb},\mathrm{ub}]$. The bottom-up approach does not produce an upper bound, and thus ub $=$ ``-'' for these tests. The fastest method for each instance is indicated in bold. }
\label{tab:smh_times}
\end{table}

\begin{table}\scriptsize\centering
\begin{tabular}{c|rrccccrrr}
    \hline
    Instance & $n$ & $m$ & $src(G)$ & $\cliquelb(G)$ & $\diam(G)$ & $dens(H)$ & $|\mathcal{P}|$ & $|\mathcal{P}|$ rem. & \% rem. \\ \hline
    BER\_2\_25\_95\_0       & 27  & 48   & 6    & 3  & 3  & 13.68 & 622    & 599    & 96.30 \\
    BER\_2\_25\_95\_1       & 27  & 50   & 5    & 2  & 2  & 14.25 & 625    & 625    & 100.00 \\
    BER\_2\_25\_95\_2       & 27  & 44   & 8    & 6  & 4  & 12.54 & 766    & 633    & 82.64 \\
    BER\_2\_25\_95\_3       & 27  & 46   & 7    & 4  & 4  & 13.11 & 695    & 590    & 84.89 \\
    BER\_2\_25\_95\_4       & 27  & 47   & 7    & 4  & 3  & 13.39 & 619    & 597    & 96.45 \\
    \hline
    BER\_2\_25\_90\_0       & 27  & 47  &  7   & 4  & 3  & 13.39 & 619    & 597    & 96.45 \\
    BER\_2\_25\_90\_1       & 27  & 47  &  7   & 4  & 3  & 13.39 & 661    & 573    & 86.69 \\
    BER\_2\_25\_90\_2       & 27  & 48  &  6   & 4  & 2  & 13.68 & 644    & 575    & 89.29 \\
    BER\_2\_25\_90\_3       & 27  & 48  &  6   & 4  & 2  & 13.68 & 644    & 575    & 89.29 \\
    BER\_2\_25\_90\_4       & 27  & 47  &  7   & 4  & 3  & 13.39 & 661    & 573    & 86.69 \\
    \hline
    BER\_2\_25\_85\_0       & 27  & 44   & 8    & 6  & 4  & 12.54 & 694    & 561    & 80.84 \\
    BER\_2\_25\_85\_1       & 27  & 42   & 9    & 8  & 4  & 11.97 & 781    & 628    & 80.41 \\
    BER\_2\_25\_85\_2       & 27  & 44   & 8    & 6  & 4  & 12.54 & 748    & 615    & 82.22 \\
    BER\_2\_25\_85\_3       & 27  & 44   & 8    & 6  & 4  & 12.54 & 766    & 633    & 82.64 \\
    BER\_2\_25\_85\_4       & 27  & 43   & 8    & 7  & 4  & 12.25 & 767    & 623    & 81.23 \\
    \hline
    BER\_2\_25\_80\_0       & 27  & 42   & 9    & 8  & 4  & 11.97 & 781    & 628    & 80.41 \\
    BER\_2\_25\_80\_1       & 27  & 42   & 9    & 8  & 4  & 11.97 & 781    & 628    & 80.41 \\
    BER\_2\_25\_80\_2       & 27  & 41   & 9    & 9  & 4  & 11.68 & 850    & 690    & 81.18 \\
    BER\_2\_25\_80\_3       & 27  & 44   & 8    & 6  & 4  & 12.54 & 748    & 615    & 82.22 \\
    BER\_2\_25\_80\_4       & 27  & 44   & 8    & 7  & 4  & 12.25 & 767    & 623    & 81.23 \\
    \hline
\end{tabular}
\caption{\footnotesize Additional information for the graphs considered in \autoref{tab:smh_times}. For each graph $G$ (from left to right) the order, size, $src(G)$, lower bound $\cliquelb(G)$, $\diam(G)$, edge density of $H(G)$ ($dens(H) := |E(H)|/\binom{|V(H)|}{2}$), number of shortest paths in $G$, number of shortest paths in $G$ remaining after variable elimination, and the percentage of shortest paths in $G$ remaining after variable elimination are indicated. }
\label{tab:smh_stats}
\end{table}

A comparison of all three computational methods for the social/infrastructure graphs and the sparse random graphs is given in \autoref{tab:times}, and additional details for these graphs are provided in \autoref{tab:stats}. The same information for the dense bipartite random graphs is given in \autoref{tab:smh_times} and \autoref{tab:smh_stats}, respectively. 
For all instances, the reported runtimes for the na\"ive model and the enhanced model both include the time to compute the heuristic solution; the bottom-up approach does not require a heuristic solution, and thus the bottom-up times do not include the heuristic runtime.

When the na\"ive and enhanced models time out, we report the best lower and upper bounds obtained by Gurobi within the timeout period (1 hour) in brackets: $[\mathrm{lb},\mathrm{ub}]$. In three instances (ieee300, enron-emails and ca-netscience), we fail to solve the root-node LP relaxation of the na\"ive model within the timeout period, and thus we report a best lower bound of ``-''. When the bottom-up approach times out, we report the best lower bound, i.e., the largest value of $K_0$ attained within the timeout period. The bottom-up method does not benefit from the use of a heuristic, and does not otherwise produce an upper bound for $src(G)$. Consequently, we report the best upper bound for the bottom-up approach as ``-''. For two instances (ieee300 and football), we were able to compute $src(G)$ using an extended time limit and additional ad hoc analysis---these values are reported for completeness in \autoref{tab:times}, and are given in parentheses.

\paragraph{Social/Infrastructure and Sparse Random Graphs} We first consider social/infrastructure graphs and sparse random graphs (\autoref{tab:times}).
In all but one of these instances (excepting those for which all three methods timed out), the bottom-up approach is the fastest method, often by a wide margin. 
This is likely due to the combination of strong, and often tight, lower bounds given by $\cliquelb(G)$ and a robust heuristic solution search provided by Gurobi. 
Because the initial lower bound $\max(\diam(G)$, $\cliquelb(G))$ is often equal to $src(G)$, in most cases optimality can be verified as soon as a single feasible solution is found. There are a few cases however, in which the bottom-up approach proves the bound $\max(\diam(G), \cliquelb(G))$ to be slack. 
Because proving infeasiblity is often computationally expensive, we expected the bottom-up approach to be outperformed by the enhanced model on these instances. However, none of the proposed methods are able to solve these instances within the hour timeout period, and the solution of the enhanced model often still terminates with a fairly large optimality gap.

In 21 of the 49 instances in \autoref{tab:times}, both the na\"ive model and the enhanced model timed out. 
In 25 of the remaining 28 instances, 
the enhanced model was faster than the na\"ive model, often by large margins. 
This result indicates the effectiveness of the variable and constraint elimination strategy introduced in \autoref{subsec:ip-imp} for sparse graphs.
Indeed, as shown in \autoref{tab:stats}, the variable elimination strategy is able to reduce the number of shortest paths in the formulation by as much as 70\%. In general, the more sparse the auxiliary graph $H$ is, the larger the number of paths that our method is able to eliminate.


One surprising observation is that despite its large number of shortest paths, and thus model sizes, the $n=150$  Watts-Strogatz instances can often be solved quite quickly. In comparison, the $n=100,  p=0.04$ Erd\H{o}s-R\'{e}nyi instances are much more difficult to compute despite having approximately 20 times fewer shortest paths. This result indicates that instance difficulty may be more closely tied to graph features other than $|\mathcal{P}|$, and partially justifies the exponential formulation of \refp{ip:master}. We also note that the edge density of $H$ is often quite low, and thus the max clique subproblems used to compute $\cliquelb(G)$
are small and particularly easy to solve in practice. 

\paragraph{Dense Random Bipartite Graphs}
Results and statistics for the dense random bipartite graphs with $n=25$ are shown in \autoref{tab:smh_times} and \autoref{tab:smh_stats}, respectively. 
In contrast with the much larger instances in \autoref{tab:times}, we see that these instances are very challenging for their size. This is particularly pronounced for the most dense instances (first 5 rows of \autoref{tab:smh_times}).
As suggested by \autoref{prop:arb-far}, these graphs have relatively large gaps between $src(G)$ and $\max(\diam(G),\cliquelb(G))$---in these instances the bound is off by as much as 50\%, much larger than the sparse instances of \autoref{tab:times}.

Because of the comparatively poor quality of the bound $\max(\diam(G),\cliquelb(G))$ for these instances, we see that the bottom-up method is often bested by the enhanced model. This degradation of performance in the bottom-up method is likely due to the larger number of iterations needed to reach a value of $K_0$ which produces a feasible solution. This trend fades as the graphs becomes slightly less dense; even at 80\% edge density (last 5 rows of \autoref{tab:smh_times}) the bottom-up method uniformly outperforms the enhanced model.
%
We note that within the context of the applications of the strong rainbow connection problem, the dense bipartite instances are somewhat artificial. They are not representative of most social/infrastructure networks which arise in practice, or even of real-world affiliation networks, as they have a part consisting of two vertices yet are highly dense. However, they do provide valuable insights into the limits of our methods when a strong lower bound on $src(G)$ is not available---particularly for the bottom-up approach. 

\section{Conclusions and Future Work}\label{sec:conclusion}
In this paper we propose the first exact computational methods for strong rainbow connection in graphs as well as a novel lower bound on the strong rainbow connection numbers of general graphs. We demonstrate that the bound we provide, especially when combined with the graph diameter lower bound, is often equal to the strong rainbow connection number in sparse graphs. We illustrate this phenomenon and compare the effectiveness of our proposed method in computational experiments which consider several random graph models and real world social and infrastructure networks. 
The method we propose is an integer program, for which we provide several computational enhancements including a random heuristic and variable/constraint elimination strategies.

Several branches of future work remain. First, it would be desirable to find an integer program formulation for computing $src(G)$ (or its variants) that contains fewer variables and/or constraints than the formulation \refp{ip:master}. Additionally, as observed in \autoref{sec:computation},
the lower bound $\max(\diam(G), \cliquelb(G))$ often exactly equals $src(G)$, in both instances generated by popular random graph models and real world networks commonly seen in the literature. As of yet, the only graphs which we have found with large $src(G) - \max ( \diam(G), \cliquelb(G) )$ are a set of complete bipartite graphs. It may be interesting to identify a set of graphs arising from other applications or random graph models in which $src(G) - \max ( \diam(G), \cliquelb(G) )$ is also large. In such graphs it is likely that the performance of the proposed bottom-up approach would be significantly degraded due to the lack of a strong initial lower bound.

\paragraph{Acknowledgments}
This work was supported by National Science Foundation grant number DMS-1720225. Additionally, D.~Mildebrath was supported by the United States Department of Defense through the National Defense Science and Engineering Graduate Fellowship program. The authors would like to thank Zachary Kingston of Rice University for helpful comments and conversations.

\renewcommand\refname{REFERENCES}
\bibliography{main}

\appendix

\section*{Appendix: Valid Inequalities and Relation to Set Packing}

In this appendix, we introduce a class of valid inequalities for the integer program \refp{ip:master}, based on the auxiliary graph $H$ constructed in \autoref{sec:lb}. In practice, these inequalities produce substantial computational benefits on only a small number of instances considered in the body of the paper. We include them here for completeness. We also prove a connection between the strong rainbow connection problem and set packing.


For any $U\subseteq{V(H)}=E(G)$ that is a clique in $H$ and any color $k\in{K}$, it follows directly from \autoref{cor:edges-different} that
\begin{equation}\label{eqn:clique-cut}
\sum_{e\in{U}}x_{ek}\leq1
\end{equation}
is valid for \refp{ip:master}. 
In this section, we prove that the inequalities \refp{eqn:clique-cut} are facet-defining for a relaxation of the feasible region of \refp{ip:master}. Specifically, we replace the equality \refp{ip:master-sum-one} with the inequality $\sum_{k\in{K}}x_{ek}\leq1$, and we eliminate the linking constraint \refp{ip:master-vub} (thereby uncoupling the $x$ and $y$ variables from the $z$ variables), to obtain the following set:
\[
    X:=\left\{
    \begin{array}{l}
    x_{ek}\in\{0,1\}\;\forall\;e\in{E(G)},\;\forall\;k\in{K}\\
    {}\\
    y_P\in\{0,1\}\;\forall\;P\in\mathcal{P}
    \end{array}
    \left|
    \begin{array}{l}
    \displaystyle\sum_{k\in{K}}x_{ek}\leq1\;\forall\;e\in{E(G)}\\
    \displaystyle\sum_{e\in{E(P)}}x_{ek}+(|P|-1)y_P\leq|P|\;\forall\;P\in\mathcal{P},\;\forall\;k\in{K}\\
    \displaystyle\sum_{P\in\mathcal{P}_{uv}}y_P\geq1\;\forall\;u\neq{v}\in{V(G)}
    \end{array}
    \right.
    \right\}.
\]
Our main result is the following:
\begin{theorem}\label{prop:facets}
For any $U\subseteq{V(H)=E(G)}$ that is a clique in $H$ and any color $k\in{K}$, the inequality \refp{eqn:clique-cut} is facet-defining for $\conv(X)$.
\end{theorem}

To prove \autoref{prop:facets}, we first derive a set-packing polytope closely related to the feasible region of \refp{ip:master}, and prove that the inequalities \refp{eqn:clique-cut} are facet-defining for the associated set-packing polytope.
We now provide a very brief summary of the set-packing polytope and some of its properties---we direct the reader to \cite{setpacking} for a comprehensive review.

Given a matrix $A\in\{0,1\}^{m\times{n}}$ without a zero row, the set-packing polytope is the set $\conv\{x\in\{0,1\}^n : Ax\leq\mathbf{1}\}$, where $\mathbf{1}\in\mathbb{R}^m$ is the vector of all $1$s. For any such 0-1 matrix $A$, define the associated graph $G(A)$ with vertex set $V(G(A))=\{v_1,\dots,v_n\}$ corresponding to the columns of $A$. Two vertices $v_i, v_j \in{V(G(A))}$ are adjacent in $G(A)$ if and only if the corresponding columns $a_i$ and $a_j$ of $A$ are \textit{not} orthogonal (i.e.~$a_i\tr{a_j^{}}\geq1$). We will make use of the following result from \cite{fulkerson1971,padberg1973}:

\begin{lemma}\label{lem:setpacking}
The inequality $\sum_{\{j:v_j\in{L}\}}x_j\leq1$ is facet-defining for the set-packing polytope with matrix $A$ if and only if $L\subseteq{V(G(A))}$ is a clique in the graph $G(A)$.
\end{lemma}

In the sequel, we will require the concept of a path fixing. 
A set $\mathcal{F}\subseteq\mathcal{P}$ is a \textit{path fixing of $G$} if $|\mathcal{F}\cap\mathcal{P}_{uv}|=1$ for all $u\neq{v}\in{V(G)}$. Let $\mathscr{F}$ denote the set of all path fixings of $G$.

Given an instance of \refp{ip:master}, and a path fixing $\mathcal{F}\in\mathscr{F}$, we define the associated set-packing polytope $S(\mathcal{F}):=\conv(X(\mathcal{F}))$ where
\[
    X(\mathcal{F}):=\left\{x_{ek}\in\{0,1\}\;\forall\;e\in{E(G)},\;\forall\;k\in{K}\left|
    \begin{array}{l}
    \displaystyle\sum_{k\in{K}}x_{ek}\leq1\;\forall\;e\in{E(G)}\\
    \displaystyle\sum_{e\in{E(P_{uv})}}x_{ek}\leq1\;\forall\;P_{uv}\in\mathcal{F},\;\forall\;k\in{K}
    \end{array}
    \right.
    \right\}.
\]
By stacking the matrix variables $(x_{ek})_{e\in{E(G)},k\in{K}}$ columnwise, we can express the constraints of the set-packing polytope $S(\mathcal{F})$ in matrix form as $B_\mathcal{F}x\leq\mathbf{1}$ where
\[
    B_\mathcal{F}:=\begin{bmatrix}
    I_m & I_m & \dots & I_m \\
    A_\mathcal{F} &     &       &     \\
        & A_\mathcal{F} &       &     \\
        &     & \ddots&     \\
        &     &       & A_\mathcal{F} \\
    \end{bmatrix}.
\]
Here, $I_m$ is the $m\times{m}$ identity matrix, and $A_\mathcal{F}$ is the matrix whose rows correspond to the elements of $\mathcal{F}$ (i.e.~paths $P_{uv}$ for $u\neq{v}\in{V(G)}$), and whose columns are edges of $G$. The entry of $A_\mathcal{F}$ corresponding to path $P$ and edge $e$ equals 1 if and only if edge $e$ is contained in path $P$. We emphasize that each column of $B_\mathcal{F}$ is indexed by a pair $(e,k)\in{E(G)\times{K}}$. The graph $G(A_\mathcal{F})$ is closely related to the auxiliary graph $H$ introduced in \autoref{sec:lb}, as detailed by \autoref{prop:clique-graph}.

\begin{proposition}\label{prop:clique-graph}
$H = \bigcap_{\mathcal{F}\in\mathscr{F}}G(A_\mathcal{F})$.
\end{proposition}
\begin{proof}
Since $V(H) = E(G)$, let the vertices of $H$ be labeled $e_1, \dots, e_m$. We assume without loss of generality that for any $\mathcal{F} \in \mathscr{F}$ the columns of $A_{\mathcal{F}}$ are ordered so that the column $a_i$ corresponds to edge $e_i$ in $E(G)$ and that the vertices of $G(A_{\mathcal{F}})$ are also labeled $e_1, \dots e_m$ so that $e_i$ corresponds to column $a_i$ for any in $e_i \in E(G(A_{\mathcal{F}}))$. Thus, $H$ and $G(A_{\mathcal{F}})$ have the same set of vertex labels for any $\mathcal{F} \in \mathscr{F}$ and so $V(H) = V(\bigcap_{\mathcal{F}\in\mathscr{F}} G(A_\mathcal{F}))$. To show that in fact $H = \bigcap_{\mathcal{F}\in\mathscr{F}} G(A_\mathcal{F})$, we next show that $E(H) = E(\bigcap_{\mathcal{F}\in\mathscr{F}} G(A_\mathcal{F}))$ by inclusions. 

To first prove $E(H) \subseteq E(\bigcap_{\mathcal{F}\in\mathscr{F}} G(A_\mathcal{F}))$, let $e_i e_j$ be an edge in $E(H)$. By construction then, there must exist a pair of vertices $u \neq v \in V(G)$ such that $e_i,e_j \in E(G)$ both separate $u$ and $v$ in $G$. Fix some $\mathcal{F} \in \mathscr{F}$. Let $k$ be the row of $A_{\mathcal{F}}$ corresponding to the $(u,v)$-path $P_{uv}$ which is fixed in $\mathcal{F}$. Since $e_i$ and $e_j$ each separate $u,v$ in $G$, both $e_i$ and $e_j$ must be in $E(P_{uv})$ and so elements $a_{k,i} = a_{k,j} = 1$ in $A_\mathcal{F}$. Since $a_i^T a_j \geq a_{k,i} a_{k,j} = 1$, we see that $e_i$ and $e_j$ are also adjacent in $G(A_\mathcal{F})$. Additionally, since $\mathcal{F}$ was chosen arbitrarily in $\mathscr{F}$, $e_i$ and $e_j$ must be adjacent in $G(A_\mathcal{F})$ for every $\mathcal{F}$ in $\mathscr{F}$. Thus, $e_i e_j$ is in $E(\bigcap_{\mathcal{F}\in\mathscr{F}} G(A_\mathcal{F}))$ so $E(H) \subseteq E(\bigcap_{\mathcal{F}\in\mathscr{F}} G(A_\mathcal{F}))$.

Next we show $E(H) \supseteq E(\bigcap_{\mathcal{F}\in\mathscr{F}} G(A_\mathcal{F}))$ by its contrapositive statement; that is, for any $e_i \neq e_j \in E(G)$ such that $e_i e_j \not \in E(H)$, it must be that $e_i e_j \not \in E(\bigcap_{\mathcal{F}\in\mathscr{F}} G(A_\mathcal{F}))$. Let $e_i \neq e_j \in E(G)$ such that $e_i e_j \not \in E(H)$. Since there does not exist a pair of vertices in $G$ which are separated by both $e_i$ and $e_j$ in $E(G)$, for each pair of vertices $u,v$ there exists a path $P_{uv}$ such that at most one of $e_i, e_j$ are contained in $E(P_{uv})$. Let $\mathcal{F}$ be a path fixing in $\mathscr{F}$ composed entirely of these paths. Consider the $k^\text{th}$ row of $A_\mathcal{F}$ and let $P$ be the path in $\mathcal{F}$ corresponding to that row. Since $P$ was chosen so that $E(P)$ does not contain both $e_i, e_j$, at least one of $a_{k,i}, a_{k,j}$ is 0. Since the index $k$ was chosen arbitrarily within the $\binom{m}{2}$ rows of $A_\mathcal{F}$, this holds for each index $k$ in $\{1, \dots, \binom{m}{2}\}$. Thus $a_i^T a_j = \sum_{k=1}^{\binom{m}{2}} a_{k,i} a_{k,j} = 0$. Since there exists at least one $\mathcal{F} \in \mathscr{F}$ such that $e_1, e_2$ are not adjacent in $G(A_\mathcal{F})$ then, the edge $e_i e_j$ is not in $E(\bigcap_{\mathcal{F}\in\mathscr{F}} G(A_\mathcal{F}))$. Having now shown that both inclusions hold, we conclude that $E(H) = E(\bigcap_{\mathcal{F}\in\mathscr{F}} G(A_\mathcal{F}))$ and thus have shown that $H = \bigcap_{\mathcal{F}\in\mathscr{F}} G(A_\mathcal{F})$.
\end{proof}

\begin{proposition}\label{prop:spp-facets}
For any path fixing $\mathcal{F}\in\mathscr{F}$, the inequalities \refp{eqn:clique-cut} are facet-defining for $S(\mathcal{F})$.
\end{proposition}
\begin{proof}
By \autoref{lem:setpacking}, it suffices to show that for any $U\subseteq{V(H)=E(G)}$ that is a clique in $H$ and $k\in{K}$, the set $U^k:=\{(e,k)\in{V(G(B_\mathcal{F}))} : e\in{U}\}$ is a clique in $G(B_\mathcal{F})$. 
To that end, fix $(e_1,k),(e_2,k)\in{U^k}$ with $e_1\neq{e_2}$.
By \autoref{prop:clique-graph}, we have that ${E(H)}\subseteq{E(G(A_\mathcal{F}))}$ and thus $(e_1,e_2)\in{E(G(A_\mathcal{F}))}$.
By inspecting the block structure of the matrix $B_\mathcal{F}$, we see that 
$(e_1,e_2)\in{E(G(A_\mathcal{F}))}$ implies
$\big((e_1,k),(e_2,k)\big)\in{E(G(B_\mathcal{F}))}$, and thus $U^k$ is a clique in $G(B_\mathcal{F})$.
\end{proof}

Next, we use \autoref{prop:spp-facets} to prove \autoref{prop:facets}. We first prove that the convex hull of the relaxation $X$ is full-dimensional. For brevity in the remainder of the section, we let $E=E(G)$.

\begin{proposition}
$\dim(\conv(X))=m|K|+|\mathcal{P}|$, i.e., $\conv(X)$ is full-dimensional.
\end{proposition}
\begin{proof}
We construct $m|K|+|\mathcal{P}|+1$ points in $X$ which are affinely independent. For $e\in{E}$ and $k\in{K}$, define a point with $x_{ek}=1$, $x_{e'k'}=0$ for all $(e',k')\in{E\times{K}}\setminus\{(e,k)\}$, and $y_P=1$ for all $P\in\mathcal{P}$. For all $P\in\mathcal{P}$, define a point with $x_{ek}=0$ for all $(e,k)\in{E\times{K}}$, $y_{P'}=1$ for all $P'\in\mathcal{P}\setminus\{P\}$, and $y_P=0$. Finally, define the single point with $x_{ek}=0$ for all $(e,k)\in{E\times{K}}$ and $y_P=1$ for all $P\in\mathcal{P}$. The resulting set of $m|K|+|\mathcal{P}|+1$ points lie in $X$ and are affinely independent.
\end{proof}

\begin{proof}[Proof of \autoref{prop:facets}]
Fix a color $\hat{k}\in{K}$, a clique in $U$ in $H$, and a path fixing $\mathcal{F}\in\mathscr{F}$. Because $\conv(X)$ is full-dimensional, it suffices to construct $m|K|+|\mathcal{P}|$ points in $X$ which satisfy \refp{eqn:clique-cut} for clique $U$ and color $\hat{k}$ with equality. 
We define these points in three sets:
\begin{itemize}
\item Set 1 ($m|K|$ points): By \autoref{prop:spp-facets}, there exist $m|K|$ points $\{\hat{x}^{(e,k)} : e\in{E},k\in{K}\}$ in $X(\mathcal{F})$ which are affinely independent and satisfy \refp{eqn:clique-cut} with equality. For each $e\in{E}$ and $k\in{K}$, define the point $z^{(e,k)}=(\hat{x}^{(e,k)},\hat{y}^{(e,k)})$ where $\hat{y}^{(e,k)}_P\in\{0,1\}$ equals $1$ if and only $P\in\mathcal{F}$. These $m|K|$ points lie in $X$ and satisfy \refp{eqn:clique-cut} with equality.
\item Set 2 ($|\mathcal{F}|$ points): Fix an edge $\hat{e}\in{U}$. For each path $P\in\mathcal{F}$ define the point $z^P=(\hat{x}^P,\hat{y}^P)$ where $\hat{x}^P_{ek}\in\{0,1\}$ equals $1$ if and only if $(e,k)=(\hat{e},\hat{k})$, and
\[
    \hat{y}^P_{P'}=\begin{cases}
    1,&P'\in\mathcal{F}\setminus\mathcal{P}_{uv}\text{ or }P'\in\mathcal{P}_{uv}\setminus\mathcal{F},\\
    0,&\text{otherwise,}
    \end{cases}
\]
where $u\neq{v}\in{V(G)}$ is the set of vertices such that $P\in\mathcal{P}_{uv}$. These $|\mathcal{F}|$ points lie in $X$ and satisfy \refp{eqn:clique-cut} with equality.
\item Set 3 ($|\mathcal{P}\setminus\mathcal{F}|$ points): Fix an edge $\hat{e}\in{U}$. For each path $P\in\mathcal{P}\setminus\mathcal{F}$, define the point $z^P=(\hat{x}^P,\hat{y}^P)$ where $\hat{x}^P_{ek}\in\{0,1\}$ equals $1$ if and only if $(e,k)=(\hat{e},\hat{k})$, and
\[
    \hat{y}^P_{P'}=\begin{cases}
    1,&P'\in\mathcal{F}\text{ or }P=P',\\
    0,&\text{otherwise.}
    \end{cases}
\]
\end{itemize}
The $m|K|+|\mathcal{P}|$ points $z^{(e,k)}$ for $(e,k)\in{E\times{K}}$ and $z^P$ for $P\in\mathcal{P}$ all lie in $X$, and satisfy \refp{eqn:clique-cut} with equality. We conclude the proof by demonstrating that they are affinely independent. With the points in set 1 we associate the multipliers $\mu^{(e,k)}$ for $e\in{E}$ and $k\in{K}$. With the points in sets 2 and 3 we associate the multipliers $\lambda^P$ for $P\in\mathcal{P}$. Suppose that
\begin{equation}\label{eqn:ai-gen}\tag{$\star$}
    \sum_{e\in{E}}\sum_{k\in{K}}\mu^{(e,k)}z^{(e,k)}+\sum_{P\in\mathcal{P}}\lambda^Pz^P=0
    \quad\text{and}\quad
    \sum_{e\in{E}}\sum_{k\in{K}}\mu^{(e,k)}+\sum_{P\in\mathcal{P}}\lambda^P=0.
\end{equation}
These conditions are equivalent to the linear system
\begin{subequations}\label{eqn:ai}
\begin{align}
\label{eqn:ai01}
\sum_{e\in{E}}\sum_{k\in{K}}\mu^{(e,k)}\hat{x}^{(e,k)}_{\hat{e},\hat{k}}+\sum_{P\in\mathcal{P}}\lambda^P&=0\\
\label{eqn:ai02}
\sum_{e\in{E}}\sum_{k\in{K}}\mu^{(e,k)}\hat{x}^{(e,k)}_{e',k'}&=0&&\forall\;(e',k')\in{E\times{K}}\setminus\{(\hat{e},\hat{k})\}\\
\label{eqn:ai03}
\sum_{e\in{E}}\sum_{k\in{K}}\mu^{(e,k)}+\sum_{P\in\mathcal{P}}\lambda^P&=\lambda^{P'}&&\forall\;P'\in\mathcal{F}\\
\label{eqn:ai04}
\lambda^P+\lambda^{P'}&=0&&\forall\;u\neq{v}\in{V(G)},P\in\mathcal{F},P'\in\mathcal{P}_{uv}\setminus\{P\}\\
\label{eqn:ai05}
\sum_{e\in{E}}\sum_{k\in{K}}\mu^{(e,k)}+\sum_{P\in\mathcal{P}}\lambda^P&=0
\end{align}
\end{subequations}
Combining \refp{eqn:ai03} and \refp{eqn:ai05}, we see immediately that $\lambda^P=0$ for all $P\in\mathcal{F}$. Plugging this into \refp{eqn:ai04}, we see that $\lambda^P=0$ for all $P\in\mathcal{P}\setminus\mathcal{F}$, and thus $\lambda^P=0$ for all $P\in\mathcal{P}$. Hence the system \refp{eqn:ai} reduces to 
\[
    \sum_{e\in{E}}\sum_{k\in{K}}\mu^{(e,k)}\hat{x}^{(e,k)}=0
    \quad\text{and}\quad
    \sum_{e\in{E}}\sum_{k\in{K}}\mu^{(e,k)}=0.
\]
Because the points $\{\hat{x}^{(e,k)}\}_{e\in{E},k\in{K}}$ are affinely independent, it follows that $\mu^{(e,k)}=0$ for all $e\in{E}$, $k\in{K}$. We have shown that \refp{eqn:ai-gen} implies $\mu=0$ and $\lambda=0$. We conclude that the points in sets 1, 2 and 3 are affinely independent.
\end{proof}

\end{document}